\documentclass[a4paper,10pt]{article}
\usepackage{amsthm,amsmath,amssymb,bbm,geometry,hyperref,enumerate,comment,nicefrac}
\usepackage[capitalise]{cleveref} 
\crefname{equation}{}{}

\newtheorem{lemma}{Lemma}[section]

\newtheorem{theorem}[lemma]{Theorem}

\newtheorem{corollary}[lemma]{Corollary}
\newtheorem{setting}[lemma]{Setting}

\newcommand{\1}{\ensuremath{\mathbbm{1}}}
\providecommand{\N}{{\ensuremath{\mathbbm{N}}}}
\providecommand{\Z}{{\ensuremath{\mathbbm{Z}}}}
\providecommand{\R}{{\ensuremath{\mathbbm{R}}}}

\providecommand{\E}{{\ensuremath{\mathbb{E}}}}
\renewcommand{\P}{{\ensuremath{\mathbb{P}}}}

\newcommand{\funcF}{F}
\newcommand{\LipConst}{L} 

\newcommand{\uniform}{\ensuremath{\mathcal{R}}}
\newcommand{\unif}{\ensuremath{\mathfrak{r}}}
\newcommand{\cost}{\ensuremath{\operatorname{Cost}}}
\newcommand{\Exp}[1]{ \E \! \left[ #1 \right]}

\newcommand{\EXP}[1]{ \E  [ #1 ]}
\newcommand{\EXPP}[1]{ \E \big[ #1 \big]}
\newcommand{\EXPPP}[1]{ \E \Big[ #1 \Big]}

\newcommand{\start}{\ensuremath{\xi}}
\newcommand{\fwBM}{\ensuremath{\xi+\mathbf{W}}}
\newcommand{\fwbm}{\ensuremath{\mathbf{W}}} 
\newcommand{\crbm}{\ensuremath{\mathbb{W}}} 
\newcommand{\RN}{\operatorname{RV}}

\title{Overcoming the curse of dimensionality in\\ 
the numerical approximation of semilinear\\ parabolic partial differential equations}

\author{Martin Hutzenthaler, Arnulf Jentzen, Thomas Kruse,\\
Tuan Anh Nguyen, and Philippe von Wurstemberger}

\begin{document}

\maketitle
\makeatletter
\let\@makefnmark\relax
\let\@thefnmark\relax
\@footnotetext{\emph{AMS 2010 subject classification:} 60H35; 65C30; 65M75}
\@footnotetext{\emph{Key words and phrases:} curse of dimensionality, high-dimensional PDEs, information based complexity, 
tractability of multivariate problems, high-dimensional semilinear BSDEs, multilevel Picard approximations, multilevel Monte Carlo method}
\makeatother

\begin{abstract}
For a long time it is well-known that high-dimensional linear parabolic partial differential equations (PDEs) can be approximated by Monte Carlo methods with a computational effort which grows polynomially both in the dimension and in the reciprocal of the prescribed accuracy. In other words, linear PDEs do not suffer from the curse of dimensionality. For general semilinear PDEs with Lipschitz coefficients, however, it remained an open question whether these suffer from the curse of dimensionality. In this paper we partially solve this open problem. More precisely, we prove in the case of semilinear heat equations with gradient-independent and globally Lipschitz continuous nonlinearities that the computational effort of a variant of the recently introduced multilevel Picard approximations grows polynomially both in the dimension and in the reciprocal of the required accuracy.
\end{abstract}

\tableofcontents

\section{Introduction and main results}

Parabolic partial differential equations (PDEs) are a fundamental tool
in applied mathematics for modelling phenomena in engineering, natural sciences, and man-made complex systems.
For instance, semilinear PDEs appear in derivative pricing models
which incorporate nonlinear risks such as default risks, interest rate risks, or liquidity risks,
and PDEs are employed to model reaction diffusion systems in chemical engineering.
The PDEs appearing in the above examples are often
high-dimensional
where the dimension corresponds to the number of financial assets such as stocks, commodities, interest rates, or exchange rates in the involved hedging portfolio.

In the literature, there exists no result which shows that
essentially any of
the high-dimensional semilinear PDEs appearing in the above mentioned applications 
can efficiently be solved approximately.
More precisely,
to the best of our knowledge,
there exists no result in the literature which shows in the case of general semilinear PDEs with
globally Lipschitz continuous coefficients that
the computational effort of an approximation algorithm grows at most polynomially
in both the PDE dimension and the reciprocal of the prescribed approximation accuracy.
In this sense no numerical algorithm is known to not suffer from the so-called curse of dimensionality,
see also the discussion after Theorem~\ref{intro_thm} below for details.

In this work we overcome the curse of dimensionality in the numerical approximation of semilinear heat equations with gradient-independent and globally Lipschitz continuous nonlinearities. As approximation algorithm we analyze a variant of the recently introduced multilevel Picard (MLP) approximations in E et al.~\cite{EHutzenthalerJentzenKruse2016}, see \eqref{intro_thm:ass2} below for the method and the paragraph after \cref{intro_thm} below for a motivation hereof.  
The main result of this article (\cref{intro_thm} below) 
shows in the case of general semilinear heat equations with gradient-independent and globally Lipschitz continuous nonlinearities
that the computational effort of the proposed approximation algorithm grows at most polynomially in both the PDE dimension $ d \in \N $ 
and the reciprocal of the required approximation accuracy $ \varepsilon > 0 $. 
More specifically, Theorem~\ref{thm1} below proves for every arbitrarily small $ \delta \in (0,\infty)$ 
that there exists $ C \in (0,\infty) $
such that for every PDE dimension $ d \in \N $ 
we have that 
the computational cost of the proposed approximation algorithm 
(see \cref{intro_thm:ass2} below) 
to achieve an approximation accuracy of size $ \varepsilon > 0 $ 
is bounded by $ C d^{1 + p(1+\delta) } \varepsilon^{-2(1+\delta)} $,
where the parameter $p \in [0,\infty)$ corresponds to the polynomial growth of the terminal 
condition and the nonlinearity of the PDE under consideration (see \cref{intro_thm} below for details).
This is essentially (up to an arbitrarily small real number $\delta \in (0,\infty)$)
the same computational complexity as the 
plain vanilla Monte Carlo algorithm 
(see, e.g., \cite{g08a,h98,h01,cdmr09,GrahamTalay2013Stochastic_simulation_and_Monte_Carlo_methods}) 
achieves in the case of \emph{linear} heat equations. 
%
%
In particular, in the language of 
information-based complexity 
this work proves, for the first time,
that general semilinear heat equations with gradient-independent 
and globally Lipschitz continuous nonlinearities
and polynomially growing terminal conditions 
are polynomially tractable 
in the setting of stochastic approximation algorithms 
(cf., e.g., Novak \& Wozniakowski~\cite[Chapter~1]{NovakWozniakowski2008} and  Novak \& Wozniakowski~\cite[Chapter~9]{NovakWozniakowski2010})
To illustrate the contribution
of this article, we now present 
in the following result, 
\cref{intro_thm} below, 
a special case of Theorem~\ref{thm1} below, which is 
the main result of article.
\begin{theorem}
\label{intro_thm}
Let $T \in (0,\infty)$, $L, p \in [0, \infty)$, $\Theta = \cup_{n = 1}^\infty \Z^n$, 
let $\xi_d \in \R^d$, $d \in \N$, satisfy 
$ \sup_{ d \in \N } \| \xi_d \|_{ \R^d } < \infty $, 
let $g_d  \colon \R^d \to \R$, $d \in \N$, and
$f_d \colon [0,T] \times \R^d \times \R \to \R$, $d \in \N$, be continuous functions which satisfy 
for all $ t \in [0, T]$, $d \in \N$, $ x \in \R^d$, $v, w \in \R$ that 
$
  |f_d(t,x,0)| 
  +  
  | g_d(x) |
\leq 
  L (1 + \| x \|_{\R^d}^p)
$
and
$
  |f_d(t, x, v) - f_d(t, x, w)|
\leq
  L
  | v - w |
$,
let $ ( \Omega, \mathcal{F}, \P ) $ be a probability space, 
let $W^{d,\theta} \colon [0,T] \times \Omega \to \R^d$, $\theta \in \Theta$, $d \in \N$, be independent standard Brownian motions,
let 
$
  \uniform^\theta \colon [0,T] \times \Omega\to \R
$, $\theta\in\Theta$, 
be i.i.d.\  continuous stochastic processes which satisfy 
for all $t\in [0,T]$, $\theta\in \Theta$ that $\uniform^\theta _t \in [t, T]$ is uniformly distributed on $[t, T]$, 
assume that
$(\uniform^{\theta})_{\theta \in \Theta}$ and $(W^{d,\theta})_{\theta \in \Theta, d \in \N}$ are independent,
let
$ 
  {U}_{ n,M}^{d, \theta} \colon [0,T] \times \R^d \times \Omega \to \R
$, $n, M\in\Z$, $\theta\in\Theta$, $d \in \N$,
satisfy
for all $d, n,M \in \N$, $\theta\in\Theta $, 
$ t \in [0,T]$, $x\in\R^d $
that $U_{-1,M}^{d, \theta}(t,x)=
U_{0,M}^{d, \theta}(t,x)=0$ and 
\begin{align}
\label{intro_thm:ass2}
&
  U_{n,M}^{d, \theta}(t,x)
  =
  \Bigg[
  \sum_{l=0}^{n-1}\tfrac{(T-t)}{M^{n-l}}\sum_{i=1}^{M^{n-l}}
  f_d \big(
    \uniform_t^{(\theta,l,i)},
    x+W_{\uniform_t^{(\theta,l,i)} - t}^{d, (\theta,l,i)}, 
    U_{l,M}^{d, (\theta,l,i)}(\uniform_t^{(\theta,l,i)},x+W_{\uniform_t^{(\theta,l,i)} - t}^{d, (\theta,l,i)} )
  \big) 
\\
&
  -
  \1_{\N}(l)
  f_d \big(
    \uniform_t^{(\theta,l,i)},
    x+W_{\uniform_t^{(\theta,l,i)} - t}^{d, (\theta,l,i)}, 
    U_{l-1,M}^{d, (\theta,-l,i)}(\uniform_t^{(\theta,l,i)},x+W_{\uniform_t^{(\theta,l,i)} - t}^{d, (\theta,l,i)})
  \big)
  \Bigg]
  +
  \sum_{i=1}^{M^n} \frac{g_d(x+W^{d, (\theta,0,-i)}_{T-t})}{M^n},
\nonumber
\end{align}
and 
for every $d, n \in \N$, $\theta \in \Theta$, $t \in [0, T]$, $x \in \R^d$ 
let $\cost_{d, n} \in \N$ be the number of realizations of scalar standard normal random variables
which are used to compute one realization of $U_{n,n}^{d, \theta}(t,x)$ (see \cref{thm1:ass1} below for a precise definition).
Then
\begin{enumerate}[(i)]
\item
for every $d \in \N$ there exists a unique at most polynomially growing continuous function $u_d \colon [0,T] \times \R^d \to \R$ 
which is a viscosity solution of
\begin{equation}
\label{intro_thm:concl1}
\begin{split}
	(\tfrac{\partial }{\partial t} u_d)(t,x) 
	+\tfrac{1}{2} 
	(\Delta_x u_d )(t,x)
	+
	f_d(t, x, u_d(t, x))
=
	0
\end{split}
\end{equation}
with
$
  u_d(T,x) = g_d(x)
$
for $ t \in (0,T) $, $ x \in \R^d $
and

\item
for every $\delta \in (0,\infty)$ there exist $n \colon \N \times (0, 1] \to \N$ and $C \in (0,\infty)$ such that
for all $d \in \N$, $\varepsilon \in (0,1]$ it holds that
$
  \cost_{d, n_{d, \varepsilon}}
\leq
  C d^{1 + p(1+\delta) } \varepsilon^{-2(1+\delta)}
$
and
\begin{equation}
\label{intro_thm:concl2}
    \big(
      \E \big[ | u_d(0, \xi_d)  -  U^{d,0}_{n_{d, \varepsilon}, n_{d, \varepsilon}} (0, \xi_d)|^2  \big]
    \big)^{\nicefrac{1}{2}}
\leq 
  \varepsilon.
\end{equation}
\end{enumerate}
\end{theorem}

\cref{intro_thm} is an immediate consequence of \cref{thm1} and Beck et al.\  \cite[Corollary~3.9]{beck2020nonlinear}. In \cref{intro_thm} and in the following presentations of this article we frequently use the standard norms on $\R^d$, $d\in \N$. In particular, we note that for all $d\in \N$, $v=(v_1,v_2,\ldots,v_d)\in \R^d$ it holds that $\|v\|_{\R^d}=\big[|v_1|^2+|v_2|^2+\ldots+|v_d|^2\big]^{1/2}$.
We now motivate the multilevel Picard approximations in \eqref{intro_thm:ass2}.
For this assume the setting of Theorem~\ref{intro_thm} and let $d\in\N$.
The Feynman-Kac formula then implies
that the exact solution $u_d$ of the PDE~\eqref{intro_thm:concl1} 
satisfies for all $t\in(0,T)$, $x\in\R^d$ that
\begin{equation}  
\label{eq:motivate_0}
\begin{split}
&
  u_d(t,x)
  =\E\!\left[g_d(x + W_{ T - t }^{d,0} )\right] 
  + \int_t^T\E\!\left[f_d(s,x+W_{s-t}^{d,0},u_d(s,x+W_{s-t}^{d,0}))\right] ds .
\end{split}     
\end{equation}
This is a fixed-point equation for $u_d$.
To this fixed-point equation
we apply the well-known Picard approximation method and a telescope sum
and let $u_{d,n}\colon[0,T]\times\R^d\to\R$, $n\in\Z$,
be
functions which satisfy for all $t\in[0,T]$, $x\in\R^d$, $n\in\N$ that
$u_{d,-1}(t,x)=u_{d,0}(t,x)=0$ and that
\begin{equation}  
\begin{split}
\label{eq:udn}
  &u_{d,n}(t,x)
  -\E\!\left[g_d(x+W_{T-t}^{d,0})\right] 
  \\&=
  \int_t^T\E\!\left[f_d(s,x+W_{s-t}^{d,0},u_{d,n-1}(s,x+W_{s-t}^{d,0}))\right]\,ds.
  \\&=
  \sum_{l=0}^{n-1}\int_t^T\E\Big[
  f_d(s,x+W_{s-t}^{d,0},u_{d,l}(s,x+W_{s-t}^{d,0}))
   -
  \1_{\N}(l)
  f_d(s,x+W_{s-t}^{d,0},u_{d,l-1}(s,x+W_{s-t}^{d,0}))
  \Big]\,ds.
\end{split}     
\end{equation}
Next we apply a multilevel Monte Carlo approach to the non-discrete expectations
and time integrals.
The crucial idea for this is that the summands on the right-hand side
of~\eqref{eq:udn} are cheap to calculate for small $l\in\N_0$
and are small for large $l\in\N_0$  since then $u_{d,l}-u_{d,l-1}$
is small. For this reason, for every $n\in\N$ we approximate the expectation
and the time integral on level $l\in\N_0$ with an average over
$M^{n-l}$ independent copies for the $n$-th approximation.
This motivates the
multilevel Picard approximations~\eqref{intro_thm:ass2}.
For more details on the derivation of the multilevel Picard approximations
see E et al.~\cite{EHutzenthalerJentzenKruse2016}.
The main difference between the approximation algorithms in~\eqref{intro_thm:ass2}
above and the approximation algorithms introduced in~\cite{EHutzenthalerJentzenKruse2016}
is that in this article we approximate time integrals by the Monte Carlo method
(inspired by~\cite{HenryLabordereOudjaneTanTouziWarin2016,Warin2018a})
instead of deterministic quadrature rules with fixed deterministic time grids
and this modification considerably simplifies the analysis
and allows us to establish \eqref{intro_thm:concl2} under merely Lipschitz continuity assumptions in the generality of Theorem~\ref{intro_thm} above. Roughly speaking, a key advantage of employing the Monte Carlo method instead of deterministic quadrature rules as in~\cite{EHutzenthalerJentzenKruse2016} is that the proposed approximation algorithms in~\eqref{intro_thm:ass2} are somehow unbiased with respect to the temporal variable in the sense that the biases of the proposed approximation algorithms in~\eqref{intro_thm:ass2} do not involve any temporal discretization error any more.

Next we relate Theorem~\ref{intro_thm} to results in the literature.
Classical deterministic methods such as finite elements or sparse grid
methods suffer from the curse of dimensionality.
Also methods based on backward stochastic differential equations
(introduced in~Pardoux \& Peng~\cite{PardouxPeng1990})
such as the Malliavin calculus based regression method
(introduced in~Bouchard \& Touzi~\cite{BouchardTouzi2004}),
the projection on function spaces method
(introduced in Gobet et al.~\cite{GobetLemorWarin2005}), cubature on Wiener space
(introduced in Crisan \& Manolarakis~\cite{CrisanManolarakis2010}),
or the Wiener chaos decomposition method 
(introduced in Briand \& Labart~\cite{BriandLabart2014})
have not been shown to
not suffer from the curse of dimensionality,
see Subsections 4.3--4.6 in E et al.~\cite{EHutzenthalerJentzenKruse2017}
for a more detailed discussion.
Moreover, recently a nested Monte Carlo method has been proposed
in Warin~\cite{Warin2018a,Warin2018b}. Simulations show that the nested
Monte Carlo method is efficient for non-large $T$ but
the method has not been shown to not suffer from the curse of dimensionality.
Branching diffusion methods
(cf., e.g., \cite{HenryLabordere2012,HenryLabordereTanTouzi2014,
HenryLabordereOudjaneTanTouziWarin2016, BouchardTanWarinZou2017})
exploit that solutions of semilinear PDEs with polynomial nonlinearities
are equal to expectations of certain functionals of
branching diffusion processes and these expectations are then
approximated by the Monte Carlo method.
Branching diffusion methods have been shown to not suffer from the curse
of dimensionality under
restrictive conditions on the initial value,
on the time horizon and on the nonlinearity; see, e.g., 
Henry-Labordere et al.~\cite[Theorem~3.12]{HenryLabordereOudjaneTanTouziWarin2016}.
If these conditions are not satisfied, then the approximations
have not been shown to not suffer from the curse of dimensionality
and simulations, e.g., for Allen-Cahn equations, indicate that the method fails to converge 
in this case.
Moreover, the multilevel Picard approximations introduced in E et al.~\cite{EHutzenthalerJentzenKruse2016} 
have been shown to not suffer from the curse of dimensionality
under very restrictive assumptions on the regularity of the exact solution;
see \cite{EHutzenthalerJentzenKruse2016,HutzenthalerKruse2017}. 
In addition, numerical simulations for 
deep learning based numerical approximation methods for PDEs 
(cf., for example, \cite{EHanJentzen2017CMStat,BeckEJentzen2017,EYu2018,
FujiiTakahashiTakahashi2017,HanJentzenE2017,HenryLabordere2017,Raissi2018,
SirignanoSpiliopoulos2017,BeckerCheriditoJentzen2018arXiv,Becketal2018}) 
indicate that such approximation methods 
seem to overcome 
the curse of dimensionality in the numerical approximation of nonlinear PDEs 
but there exist no rigorous mathematical results which demonstrate this conjecture. 
To the best of our knowledge, 
the scheme~\cref{intro_thm:ass2} in \cref{intro_thm} above
is the first numerical approximation scheme 
in the scientific literature for which it has been proven that it overcomes the curse of 
dimensionality in the numerical approximation of 
general gradient-independent semilinear heat PDEs.

After the preprint version of this article has been published, several follow-up research articles which are based on this work have appeared.
In particular, we refer to \cite{hutzenthaler2019bovercoming} for MLP approximations of the form \eqref{intro_thm:ass2} in the case of semilinear Kolmogorov PDEs involving a second order differential operator with varying coefficients instead of just the Laplacian, we refer to \cite{beck2019overcoming} for MLP approximations of the form \eqref{intro_thm:ass2} in the case of semilinear PDEs with non-globally Lipschitz continuous nonlinearities, we refer to \cite{hutzenthaler2019overcoming} for MLP approximations of the form \eqref{intro_thm:ass2} in the case of semilinear PDEs with gradient-dependent nonlinearities, we refer to \cite{beck2020overcoming} for MLP approximations of the form \eqref{intro_thm:ass2} in the case of semilinear elliptic PDEs, and we also refer to \cite{becker2020numerical} for several numerical simulations for MLP approximations of the form \eqref{intro_thm:ass2} in the case of Allen-Cahn PDEs (see \cite[Subsection~3.1]{becker2020numerical}), in the case of Sine-Gordon type PDEs (see \cite[Subsection~3.2]{becker2020numerical}), in the case of systems of semilinear heat PDEs (see \cite[Subsection~3.3]{becker2020numerical}), and in the case of semilinear Black-Scholes PDEs (see \cite[Subsection~3.4]{becker2020numerical}).

The remainder of this article is organized as follows. In Section~\ref{sec:semi-norm} we introduce a family of suitable semi-norms for a certain class of random fields and we also reveal several basic properties of these semi-norms. Note that the exact solutions $u_d\colon [0,T]\times \R^d \to \R$, $d\in \N$, of the PDEs in~\eqref{intro_thm:concl1} are deterministic functions while the numerical approximations ${U}_{ n,M}^{d, \theta} \colon [0,T] \times \R^d \times \Omega \to \R$, $n, M\in\Z$, $\theta\in\Theta$, $d \in \N$, in \eqref{intro_thm:ass2} are random fields on $[0,T] \times \R^d$. The in Section~\ref{sec:semi-norm} introduced semi-norms for random fields are exploited to estimate the difference between the exact solutions of the PDEs in~\eqref{intro_thm:concl1} and the numerical approximations proposed in this work. In Section~\ref{b34} we subsequently develop the overall complexity analysis for the proposed numerical approximation algorithms to establish Theorem~\ref{thm1} in Subsection~\ref{t5} below and, thereby, to prove Theorem~\ref{intro_thm} in this introductory section. More formally, in Subsection~\ref{t1} we formulate the MLP approximation algorithms proposed in this work and the framework which we employ in our error analysis for the proposed MLP approximation algorithms.
In Subsection~\ref{t2} we establish several basic properties of the proposed MLP approximation algorithms and in Subsection~\ref{t3} we prove a priori estimates for the exact solutions of the PDEs under consideration. Our error analysis for the proposed MLP approximation algorithms can be found in Subsection~\ref{t4}. In Subsection~\ref{t5} we combine this error analysis with a computational cost analysis for the proposed MLP approximation algorithms to accomplish the overall complexity analysis for the proposed MLP approximation algorithms.

\section{Analysis of semi-norms}\label{sec:semi-norm}
In this section we introduce in Subsection~\ref{s1} a family of suitable semi-norms for a certain class of random fields. In Subsection~\ref{s2} we formulate a few basic consequences of Fubini's theorem. In Subsection~\ref{s3} we establish several basic properties of the in Subsection~\ref{s1} introduced semi-norms which we employ in our error analysis for the proposed MLP approximation algorithms in Section~\ref{b34} below.

\subsection{Setting}\label{s1}
Throughout this section we frequently consider the following setting.
\begin{setting}
\label{a00}
Let $d\in \N$, $T \in (0,\infty)$, $L \in [0,\infty)$, $\start\in \R^d$,
let 
$F\colon C([0, T] \times \R^d, \R) \to C([0, T] \times \R^d, \R) $
satisfy for all $u,v\in C([0, T] \times \R^d, \R)$, $t\in [0,T]$, 
$x\in  \R^d$ that
\begin{equation}  \begin{split}
\label{a01}
    |(F(u))(t,x) -(F(v))(t,x)|\leq 
    \LipConst\left|u(t,x)-v(t,x)\right|,
\end{split}     \end{equation}
let
$(\Omega, \mathcal{F}, \P)$
be a probability space,
let $\fwbm \colon [0,T] \times \Omega \to \R^d$
be a standard Brownian motion
with continuous sample paths, 
and
for every $k \in \N_0$ and 
every $(\mathcal{B}([0, T] \times \R^d) \otimes \mathcal{F}) / \mathcal{B}(\R)$-measurable function 
$V \colon [0, T] \times \R^d \times \Omega \to \R$ let
$
\left\|V\right\|_{k} \in [0,\infty]
$
be the extended real number given by
\begin{align}
\begin{aligned}
  \left\|V\right\|_{k}^2
  =
  \begin{cases}\displaystyle
 \E\!\left[\left|V(0,\start)\right|^2\vphantom{\big|}\right]&\qquad 
\colon  k= 0\\[10pt] \displaystyle
   \frac{1}{T^{k}} \int_0^{T}
\tfrac{t^{k-1}}{(k-1)!}\,  \E\!\left[
  \left|V(t,\fwBM_t)\right|^2\vphantom{\big|}\right]dt&\qquad 
\colon
  k\geq 1.
  \end{cases}
\end{aligned}\label{a03}
\end{align}
\end{setting}

Observe that Setting~\ref{a00} specifies in \eqref{a03} for every $k \in \N_0$ and 
every $(\mathcal{B}([0, T] \times \R^d) \otimes \mathcal{F}) / \mathcal{B}(\R)$-measurable function 
$V \colon [0, T] \times \R^d \times \Omega \to \R$ the quantity $\|V\|_k$. Note that for every topological space $(X,\mathcal X)$ it holds that the Borel sigma-algebra $\mathcal B(X)$ is the smallest sigma-algebra that contains $\mathcal X$. 

\subsection{Expectations of random fields}\label{s2}

In this subsection we formulate in Lemma~\ref{prodRF_eval_nonneg}, Lemma~\ref{contRF_eval_nonneg}, 
Lemma~\ref{contRF_eval_integrable}, and Corollary~\ref{id_RF} below 
some elementary consequences of Fubini's theorem.
For the formulations of Lemma~\ref{prodRF_eval_nonneg}, Lemma~\ref{contRF_eval_nonneg}, and
Lemma~\ref{contRF_eval_integrable}  we recall that for every probability space $ ( \Omega, \mathcal{F}, \P ) $, every measurable space $ ( S , \mathcal{S} ) $, every $ \mathcal{F} $/$ \mathcal{S} $-measurable function $Y \colon \Omega \to S$, and every $A\in \mathcal S$ it holds that $(Y ( \P )_{ \mathcal{S} })(A)=\P(Y\in A)=\P(Y^{-1}(A))$ (pushforward measure).

\begin{lemma}
\label{prodRF_eval_nonneg}
Let $ ( \Omega, \mathcal{F}, \P ) $ be a probability space, 
let $\mathcal{G} \subseteq \mathcal{F}$ be a sigma-algebra on $ \Omega $,
let $ ( S , \mathcal{S} ) $ be a measurable space, 
let  $ U =  (U(s))_{s \in S} = (U(s, \omega))_{s \in S, \omega \in \Omega} \colon S \times \Omega \to [0,\infty)$
be an $(\mathcal{S} \otimes \mathcal{G}) / \mathcal{B}([0,\infty))$-measurable function, 
let $Y \colon \Omega \to S$ be a $ \mathcal{F} $/$ \mathcal{S} $-measurable function, 
assume that 
$Y$ 
and 
$
  \mathcal{G}
$
are independent,
and let $\Phi \colon S \to [0,\infty]$ satisfy
for all $s \in S$ that
$
  \Phi(s) = \Exp{U(s)}
$.
Then
\begin{enumerate}[(i)]
\item \label{prodRF_eval_nonneg:item1}
it holds that $U(Y) = ( \Omega \ni \omega \mapsto U(Y(\omega), \omega) \in [0,\infty))$ is an $\mathcal{F} / \mathcal{B}([0,\infty))$-measurable function
and

\item \label{prodRF_eval_nonneg:item2}
it holds that
$
  \Exp{U(Y)} 
= 
  \Exp{\Phi(Y)} 
= 
  \int_{S}  \Exp{U(s)} (Y ( \P )_{ \mathcal{S} })(ds)
$.
\end{enumerate}
\end{lemma}

\begin{lemma}
\label{contRF_eval_nonneg}
Let $ ( \Omega, \mathcal{F}, \P ) $ be a probability space, 
let $ ( S , \delta ) $ be a separable metric space, 
let  $ U =  (U(s))_{s \in S} = (U(s, \omega))_{s \in S, \omega \in \Omega} \colon S \times \Omega \to [0,\infty)$
be a continuous random field, 
let $Y \colon \Omega \to S$ be a random variable, 
assume that $U$ and $Y$ are independent,
and let $\Phi \colon S \to [0,\infty]$ satisfy
for all $s \in S$ that
$
  \Phi(s) = \Exp{U(s)}
$.
Then
\begin{enumerate}[(i)]
\item \label{contRF_eval_nonneg:item1}
it holds that $U(Y) = ( \Omega \ni \omega \mapsto U(Y(\omega), \omega) \in [0,\infty))$ is an $\mathcal{F} / \mathcal{B}([0,\infty))$-measurable function
and

\item \label{contRF_eval_nonneg:item2}
it holds that
$
  \Exp{U(Y)} 
= 
  \Exp{\Phi(Y)} 
= 
  \int_{S}  \Exp{U(s)} (Y ( \P )_{\mathcal{B}(S)})(ds)
$.
\end{enumerate}
\end{lemma}


\begin{lemma}
\label{contRF_eval_integrable}
Let $ ( \Omega, \mathcal{F}, \P ) $ be a probability space, 
let $ ( S , \delta ) $ be a separable metric space, 
let  $ U =  (U(s))_{s \in S} = (U(s, \omega))_{s \in S, \omega \in \Omega} \colon S \times \Omega \to \R$
be a continuous random field, 
let $Y \colon \Omega \to S$ be a random variable, 
assume that $U$ and $Y$ are independent, 
and assume that 
$
  \int_{S}  \Exp{|U(s)|} (Y ( \P )_{\mathcal{B}(S)})(ds) < \infty
$.
Then
\begin{enumerate}[(i)]
\item \label{contRF_eval_integrable:item1}
it holds that $U(Y) = ( \Omega \ni \omega \mapsto U(Y(\omega), \omega) \in \R)$ is an $\mathcal{F} / \mathcal{B}(\R)$-measurable function
and

\item \label{contRF_eval_integrable:item2}
it holds that $\Exp{|U(Y)|} < \infty$ and 
$
  \Exp{U(Y)} 
= 
  \int_{S}  
    \mathbbm{1}_{ \{ \mathfrak{s} \in S \colon \Exp{|U(\mathfrak{s})|} < \infty \} } (s) \,
    \Exp{U(s)} 
  (Y ( \P )_{\mathcal{B}(S)})(ds)
$.
\end{enumerate}
\end{lemma}

\begin{corollary}
\label{id_RF}
Let $ ( \Omega, \mathcal{F}, \P ) $ be a probability space, 
let $ ( S , \delta ) $ be a separable metric space, 
let $(E, \mathcal{E})$ be measurable space,
let  $ U_1, U_2  \colon S \times \Omega \to \R$
be continuous random fields, 
let $Y_1, Y_2 \colon E \times \Omega \to S$ be random fields, 
assume for all $i \in \{ 1, 2 \}$ that $U_i$ and $Y_i$ are independent, 
assume that $U_1$ and $U_2$ are identically distributed, and
assume that $Y_1$ and $Y_2$ are identically distributed.
Then
it holds that
$U_1(Y_1) =  (E \times \Omega \ni (e, \omega) \mapsto U_1(Y_1(e), \omega) \in \R)$ 
and
$U_2(Y_2) = (E \times \Omega \ni (e, \omega) \mapsto U_2(Y_2(e), \omega) \in \R)$ 
are identically distributed random fields.
\end{corollary}

\subsection{Properties of the semi-norms}\label{s3}

In this subsection we establish 
in \cref{a07a}, \cref{a07}, \cref{a08}, \cref{a14}, \cref{a26}, 
and \cref{a36} a few basic properties for the quantities in \eqref{a03} 
in \cref{a00} above.
The proof of \cref{a07a} is clear and therefore omitted.

\begin{lemma}[Semi-norm property]
\label{a07a}
Assume \cref{a00}, let $k\in\N_0$, $\lambda \in \R$, and let $U, V \colon [0, T] \times \R^d \times \Omega \to \R$ be $(\mathcal{B}([0, T] \times \R^d) \otimes \mathcal{F}) / \mathcal{B}(\R)$-measurable functions.
Then  
\begin{enumerate}[(i)]
\item \label{a07a:item1} 
it holds that
$\| U + V \|_{k}\leq \|U\|_{k}+\|V\|_{k}$ and
\item \label{a07a:item2} 
it holds that
$\|\lambda U \|_{k}= |\lambda| \| U \|_k$.
\end{enumerate}
\end{lemma}

\begin{lemma}[Expectations]\label{a07}
Assume \cref{a00}, let $k\in\N_0$, 
let $U \colon [0, T] \times \R^d \times \Omega \to \R$ be a continuous random field, 
assume that $U$ and $\fwbm$ are independent, 
and assume for all $t\in [0,T]$, $x\in  \R^d$ that $\E[|U(t,x)|]<\infty$.
Then it holds that
\begin{equation}  
\label{a07:concl1}
    \left\|  
      [0,T] \times \R^d\times \Omega\ni (t,x,\omega) \mapsto \E[U(t, x)] \in \R
    \right\|_{k}
    =
    \left\|\E[U]\right\|_{k}
    \leq
    \left\|U\right\|_{k}.
\end{equation}
\end{lemma}

\begin{proof}[Proof of \cref{a07}]
Throughout this proof let $v \colon [0, T] \times \R^d \to  \R $ satisfy 
for all $t\in [0,T]$, $x\in  \R^d$ that
$
  v(t, x)
=
  \E[U(t, x)]
$
and let
$\mu_{t} \colon \mathcal{B}(\R^d) \to [0,1]$, $t \in [0,T]$, be the probability measures which satisfy 
for all $t \in [0,T]$, $B \in \mathcal{B}(\R^d)$ that
$
  \mu_t(B) 
=
  \P( \start + \fwbm_t \in B )
$.
Note that Jensen's inequality and \eqref{a03} assure that
\begin{equation}
\label{a07:eq0}
\begin{split}
  \left\|\E[U]\right\|_{0}^2
=
  \left\| v \right\|_{0}^2
=
  \Exp{ | v(0, \start) |^2  }
=
  | v(0, \start) |^2
=  
  | \Exp{ U(0, \start) } |^2
\leq
  \Exp{ | U(0, \start) |^2 }
=
  \left\| U \right\|_{0}^2.
\end{split}
\end{equation}
Next observe that \eqref{a03} ensures that
for all $l \in \N$ it holds that
\begin{equation}
\label{a07:eq1}
\begin{split}
  \left\|\E[U]\right\|_{l}^2
=
  \left\| v \right\|_{l}^2
=
  \frac{1}{T^{l}} 
  \int_0^{T}
    \tfrac{t^{l-1}}{(l-1)!}\,  
    \Exp{ \left| v(t, \xi + \fwbm_t)\right|^2 }
  dt.
\end{split}
\end{equation}
Moreover, note that the integral transformation theorem, Jensen's inequality, \cref{contRF_eval_nonneg}, the hypothesis that $U$ is a continuous random field, and the hypothesis that $U$ and $\fwbm$ are independent ensure that
for all $t \in [0, T]$ it holds that
\begin{equation}
\begin{split}
  \Exp{ \left| v(t, \xi + \fwbm_t)\right|^2 }
&=
  \int_{\R^d} \left| v(t, x)\right|^2 \, \mu_t( d x )
=
  \int_{\R^d} \left| \E[U(t, x)] \right|^2 \, \mu_t( d x ) \\
&\leq
  \int_{\R^d} \Exp{ \left| U(t, x) \right|^2 } \, \mu_t( d x )
=
  \Exp{ \left| U(t, \xi + \fwbm_t) \right|^2 }.
\end{split}
\end{equation}
This and \eqref{a07:eq1} imply that 
for all $l \in \N$ it holds that
\begin{equation}
\begin{split}
  \left\|\E[U]\right\|_{l}^2
\leq
  \frac{1}{T^{l}} 
  \int_0^{T}
    \tfrac{t^{l-1}}{(l-1)!}\,  
    \Exp{ \left| U(t, \xi + \fwbm_t) \right|^2 }
  dt
=
  \left\| U \right\|_{l}^2.
\end{split}
\end{equation}
Combining this and \eqref{a07:eq0} establishes \eqref{a07:concl1}.
The proof of \cref{a07} is thus completed.
\end{proof}
\begin{lemma}[Linear combinations of i.i.d.\ random variables]\label{a08}
Assume \cref{a00}, let $k\in\N_0$, $n\in\N$,  $r_1,\ldots,r_n\in\R$, 
let
$U_1, \ldots,U_n \colon [0, T] \times \R^d \times \Omega \to \R$ be continuous i.i.d.\ random fields,
assume that
$(U_i)_{i \in \{1, 2, \ldots, n\}}$ and $\fwbm$ are independent, 
and assume  
for all $t \in [0,T]$, $x\in \R^d$ that
$\E[|U_1(t,x)|]<\infty$.
Then it holds that
\begin{align}
    \left\|\sum_{i=1}^nr_i\left(U_i-\E[U_i]\right)\right\|_{k}
    &= \left\|U_1-\E[U_1]\right\|_{k}
\left[\sum_{i=1}^n|r_i|^2\right]^{\!\nicefrac{1}{2}}    \leq \left\|U_1\right\|_{k}
\left[\sum_{i=1}^n|r_i|^2\right]^{\!\nicefrac{1}{2}}.\label{a09}
\end{align}     
\end{lemma}

\begin{proof}[Proof of Lemma~\ref{a08}]
Throughout this proof let 
$\mathcal{G} \subseteq \mathcal{F}$ satisfy that $\mathcal{G} = \sigma_{\Omega}((U_i)_{i \in \{1, 2, \ldots, n\}})$,
let
$v_i \colon [0, T] \times \R^d \times \Omega \to \R$, $i \in \{1, 2, \ldots, n \}$, satisfy 
for all $i \in \{1, 2, \ldots, n \}$, $t \in [0,T]$, $x\in \R^d$ that
$
  v_i(t, x)
=
  U_i(t, x) - \EXP{ U_i(t, x)}
$,
and let $\mu_{t} \colon \mathcal{B}(\R^d) \to [0,1]$, $t \in [0,T]$, be the probability measures which satisfy 
for all $t \in [0,T]$, $B \in \mathcal{B}(\R^d)$ that
$
  \mu_t(B) 
=
  \P( \start + \fwbm_t \in B )
$.
Note that the fact that $U_1, \ldots, U_n$ are continuous random fields, Beck et al.~\cite[Lemma 2.4]{Becketal2018}, and Fubini's theorem imply that 
for every $i \in \{1, 2, \ldots, n \}$ it holds that 
$v_i$ is a $(\mathcal{B}([0, T] \times \R^d) \otimes \mathcal{G} ) / \mathcal{B}(\R)$-measurable function.
The hypothesis that $\mathcal{G}$ and $\fwbm$ are independent, \cref{prodRF_eval_nonneg}, 
the fact that 
for all $t \in [0,T]$, $x\in \R^d$ it holds that 
$v_1(t, x), v_2(t, x), \ldots,  v_n(t, x)$ are i.i.d.\ random variables with 
$\Exp{ | v_1(t, x)|} < \infty$ and $\Exp{  v_1(t, x)} = 0$,
and Klenke~\cite[Theorem~5.4]{k08b}
therefore demonstrate that
for all $t \in [0, T]$ it holds that
\begin{equation}
\label{a08:eq1}
\begin{split}
  \Exp{ \big| 
    \textstyle \sum_{i=1}^n
      r_i v_i(t, \xi + \fwbm_t)
  \big|^2 }
&=
  \int_{\R^d}
    \Exp{
      \big| 
        \textstyle \sum_{i=1}^n
          r_i v_i(t, x)
      \big|^2 
    }
  \, \mu_t( d x ) =
  \int_{\R^d}
    \textstyle \sum_{i,j=1}^n
    \Exp{\,
          r_i r_j v_i(t, x) v_j(t, x)
    }
  \, \mu_t( d x ) \\
&=
  \int_{\R^d}
      \textstyle \sum_{i=1}^n
        | r_i |^2 \, \Exp{ |  v_i(t, x) |^2 }
  \, \mu_t( d x )  =
  \textstyle \sum_{i=1}^n \left(
    | r_i |^2 
    \int_{\R^d}
      \Exp{ |v_i(t, x) |^2 }
    \, \mu_t( d x )  
  \right) \\
&=
  \Big[ \textstyle \sum_{i=1}^n  | r_i |^2 \Big]
    \int_{\R^d}
      \Exp{ |v_1(t, x) |^2 }
    \, \mu_t( d x )  =
  \Big[ \textstyle \sum_{i=1}^n  | r_i |^2 \Big]
  \Exp{ |v_1(t, \xi + \fwbm_t) |^2 }.
\end{split}
\end{equation}
This and \eqref{a03} imply that
\begin{equation}
\label{a08:eq2}
\begin{split}
  \big\| \textstyle \sum_{i=1}^n r_i\left(U_i-\E[U_i]\right) \! \big\|_{0}^2
&=
  \big\| \textstyle \sum_{i=1}^n r_iv_i\big\|_{0}^2
=
  \Exp{\big| \textstyle \sum_{i=1}^n r_i v_i (0, \xi) \big|^2} =
  \Big[ \textstyle \sum_{i=1}^n  | r_i |^2 \Big]
  \Exp{ |v_1(0, \xi ) |^2 }\\
&
=
  \Big[ \textstyle \sum_{i=1}^n  | r_i |^2 \Big]
  \| v_1 \|_0^2 =
  \Big[ \textstyle \sum_{i=1}^n  | r_i |^2 \Big]
  \| U_1 - \Exp{U_1} \! \|_0^2.
\end{split}
\end{equation}
Moreover, observe that \eqref{a03} and \eqref{a08:eq1} show that 
for all $l \in \N$ it holds that
\begin{equation}
\label{a08:eq3}
\begin{split}
  \big\| \textstyle \sum_{i=1}^n r_i\left(U_i-\E[U_i]\right) \! \big\|_{l}^2
&=
  \big\| \textstyle \sum_{i=1}^n r_iv_i\big\|_{l}^2 =
  \frac{1}{T^{l}} 
  \int_0^{T}
    \tfrac{t^{l-1}}{(l-1)!}\,  
    \Exp{ \big| \textstyle \sum_{i=1}^n r_iv_i (t, \xi + \fwbm_t) \big|^2 }
  dt \\
&=
  \frac{1}{T^{l}} 
  \int_0^{T}
    \tfrac{t^{l-1}}{(l-1)!}\,  
    \Big[ \textstyle \sum_{i=1}^n  | r_i |^2 \Big]
    \EXPP{ |v_1(t, \xi + \fwbm_t) |^2 }
  dt \\
&=
  \Big[ \textstyle \sum_{i=1}^n  | r_i |^2 \Big]
  \| v_1 \|_l^2 
=
  \Big[ \textstyle \sum_{i=1}^n  | r_i |^2 \Big]
  \| U_1 - \Exp{U_1} \! \|_l^2.
\end{split}
\end{equation}
Next observe that \eqref{a03} assures that
\begin{equation}
\begin{split}
\label{a08:eq4}
  \left\|U_1-\E[U_1]\right\|_0^2 
&= 
  \| v_1\|_0^2
=
  \E\!\left[| v_1(0,\start)|^2\right]
=
  \E\!\left[|U_1(0,\start)-\E\!\left[U_1(0,\start)\right]|^2\right] \\
&= 
  \E\!\left[|U_1(0,\start)|^2\right] - \left| \E\!\left[U_1(0,\start)\right] \right|^2 
\leq 
  \E\!\left[|U_1(0,\start)|^2\right]
= 
  \|U_1\|_0^2.
\end{split}
\end{equation}
Furthermore, note that the hypothesis that $\mathcal{G}$ and $\fwbm$ are independent and \cref{prodRF_eval_nonneg} assure that
for all $t \in [0, T]$ it holds that
\begin{equation}
\begin{split}
  \Exp{ |v_1(t, \xi + \fwbm_t) |^2 }
&=
  \int_{\R^d}
    \Exp{ |v_1(t, x) |^2 }
  \, \mu_t( d x ) =
  \int_{\R^d}
    \Exp{ |U_1(t, x) - \Exp{U_1(t, x)} |^2 }
  \, \mu_t( d x ) \\
&=
  \int_{\R^d}
     \Exp{ |U_1(t, x)|^2} - | \EXP{U_1(t, x)} |^2
  \, \mu_t( d x )\\
& \leq
  \int_{\R^d}
   \Exp{ | U_1(t, x) |^2 }
  \, \mu_t( d x ) 
=
  \Exp{ | U_1(t, \xi + \fwbm_t) |^2 }.
\end{split}
\end{equation}
This and \eqref{a03} demonstrate that 
for all $l \in \N$ it holds that
\begin{equation}
\begin{split}
  \| U_1-\E[U_1] \|_{l}^2
=
  \| v_1 \|_{l}^2
&=
  \frac{1}{T^{l}} 
  \int_0^{T}
    \tfrac{t^{l-1}}{(l-1)!}\,  
    \Exp{ |v_1(t, \xi + \fwbm_t) |^2 }
  dt \leq
  \frac{1}{T^{l}} 
  \int_0^{T}
    \tfrac{t^{l-1}}{(l-1)!}\,  
    \Exp{ | U_1(t, \xi + \fwbm_t) |^2 }
  dt
=
  \| U_1 \|_l^2.
\end{split}
\end{equation}
Combining this, \eqref{a08:eq2}, \eqref{a08:eq3}, and \eqref{a08:eq4} establishes that
\begin{equation}
  \big\| \textstyle \sum_{i=1}^nr_i\left(U_i-\E[U_i]\right) \! \big\|_{k}^2
= 
  \Big[ \textstyle\sum_{i=1}^n |r_i|^2 \Big]
  \left\|U_1-\E[U_1]\right\|_{k}^2
\leq 
  \Big[ \textstyle\sum_{i=1}^n|r_i|^2 \Big]
  \left\|U_1\right\|_{k}^2.
\end{equation}
This completes the proof of \cref{a08}.
\end{proof}

\begin{lemma}[Lipschitz property of $F$]
\label{a14}
Assume \cref{a00}, 
let $k\in\N_0$, 
and let $U,V \colon [0, T] \times \R^d \times \Omega \to \R$ be continuous random fields. 
Then 
\begin{enumerate}[(i)]
\item \label{a14:item1}
it holds that 
$
  F(U)
= 
  \big( 
    [0, T] \times \R^d \times \Omega  \ni (t, x, \omega) \mapsto 
     \big[ 
       F \big( [0, T] \times \R^d \ni (s, z) \mapsto U(s, z, \omega) \in \R \big) 
     \big] (t, x) 
     \in \R
  \big)
$
is a continuous random field 
and
\item \label{a14:item2}
it holds that
$
\|F(U)-F(V)\|_{k}
  \leq 
\LipConst\|U-V\|_{k}
$.
\end{enumerate}
\end{lemma}

\begin{proof}[Proof of \cref{a14}]
Throughout this proof let $\pi_{t, x} \colon C([0, T] \times \R^d, \R) \to \R$, $t \in [0, T]$, $x \in \R^d$, satisfy 
for all $t \in [0, T]$, $x \in \R^d$, $v \in C([0, T] \times \R^d, \R)$ that
$
  \pi_{t, x} (v) = v(t, x)
$
and let $\mathfrak{U} \colon \Omega \to C([0, T] \times \R^d, \R)$ satisfy 
for all $\omega \in \Omega$, $t \in [0, T]$, $x \in \R^d$ that 
$
  (\mathfrak{U}(\omega)) (t, x) 
=
  U(t, x, \omega)
$.
Note that
for all $t \in [0, T]$, $x \in \R^d$, $\omega \in \Omega$ it holds that
$
  (F(U))(t, x, \omega) 
=
  \big[ F(\mathfrak{U}(\omega)) \big] (t, x)
=
  \pi_{t, x} \big[ F(\mathfrak{U}(\omega))\big]
$.
Hence, we obtain 
for all $t \in [0, T]$, $x \in \R^d$ that
$
  \big(
    \Omega \ni \omega \mapsto (F(U))(t, x, \omega) 
  \big)
=
  \pi_{t, x} \circ  F \circ \mathfrak{U}
$.
The fact that $\mathfrak{U}$ is $\mathcal{F} $/$ \mathcal{B}(C([0, T] \times \R^d, \R))$-measurable,
the fact that for all $t \in [0, T]$, $x \in \R^d$ it holds that $\pi_{t, x}$ is $\mathcal{B}(C([0, T] \times \R^d, \R)) $/$ \mathcal{B}(\R)$-measurable, 
and 
the fact that $F$ is 
$\mathcal{B}(C([0, T] \times \R^d, \R)) $/$ \mathcal{B}(C([0, T] \times \R^d, \R))$-measurable (cf.\ \eqref{a01}) 
hence assure that for all $t \in [0, T]$, $x \in \R^d$ it holds that 
$ 
  (
    \Omega \ni \omega \mapsto (F(U))(t, x, \omega) 
  )
$
is $ \mathcal{F} $/$ \mathcal{B}(\R)$-measurable.
Combining this with the fact that 
for all $v \in C([0, T] \times \R^d, \R)$
it holds that $F(v) \in C([0, T] \times \R^d, \R)$ demonstrates that $F(U)$ is a continuous random field. 
This establishes Item~\eqref{a14:item1}.
Next observe that \eqref{a01} and \eqref{a03} show that
\begin{equation}
\label{a18}
\begin{split}
\left\|F(U)-F(V)\right\|_{0}^2&=
\E\!\left[\left|(F(U)-F(V))(0,\start)\right|^2\right] \leq  
\E\!\left[ \LipConst^2 
\left|(U-V)(0,\start)\right|^2\right]=\LipConst^2\left\|U-V\right\|_{0}^2.
\end{split}
\end{equation}
Moreover, note that  \eqref{a01} and \eqref{a03} imply that 
for all $l \in \N$ it holds that
\begin{equation}
\begin{split}
\|F(U)-F(V)\|_{l}^2&=
\frac{1}{T^l}\int_0^{T}\tfrac{t^{l-1}}{(l-1)!}\,\E\!\left[\left|\big(F(U)-F(V)\big)(t,\fwBM_{t}))\right|^2\right]\,dt
\\
 & \leq
  \frac{1}{T^l}\int_0^{T}\tfrac{t^{l-1}}{(l-1)!}\,
  \E\!\left[\LipConst^2\big|(U-V)(t,\fwBM_{t}))\big|^2\right]\,dt
=  L^2
    \|U-V\|_{l}^2.
\end{split}
\end{equation} 
Combining this and \eqref{a18} establishes Item~\eqref{a14:item2}. 
The proof of \cref{a14} is thus completed.
\end{proof}

\begin{lemma}[Monte Carlo time integrals]\label{a26}
Assume \cref{a00},
let $k\in\N_0$,
let $U \colon [0, T] \times \R^d \times \Omega \to \R$ be a continuous random field, 
let $ \unif\colon \Omega\to [0,1]$ be a $\mathcal{U}_{[0,1]}$-distributed random variable, 
let $\uniform\colon [0,T]\times \Omega\to \R$ satisfy 
for all $t\in [0,T]$ that 
$\uniform _t = t+ (T-t)\unif$,
let $\crbm\colon [0,T]\times\Omega\to\R^d$ be a standard Brownian motion with continuous sample paths, 
and assume that $U, \fwbm, \unif$, and $\crbm$ are independent.   
Then it holds that
\begin{align}  
\label{a27}
  \left\| [0,T]\times\R^d\times \Omega\ni (t,x,\omega) \mapsto (T-t) \big[U(\uniform_t,x+ \crbm_{\uniform_t}-\crbm _{t})\big](\omega)\in\R\right\|_{k}\leq T \|U\|_{k+1}.
\end{align}
\end{lemma}

\begin{proof}[Proof of Lemma~\ref{a26}]
Throughout this proof let 
$V^{(t)} = (V^{(t)}_s(\omega))_{s \in [t, T], \omega \in \Omega} \colon [t, T] \times  \Omega \to \R$, $t \in [0, T]$, 
be the random fields which satisfy 
for all $t \in [0, T]$, $s \in [t, T]$ that
$
  V^{(t)}_s
=
  U(s,\xi+ \fwbm_t + \crbm_{s}-\crbm_{t})
$.
Observe that the fact that $\fwbm$, $\crbm$, and $U$ are independent, the hypothesis that $U$ is a continuous random field, Lemma~\ref{contRF_eval_nonneg}, and the fact that 
for all $t \in [0, T]$, $s \in [t, T]$ it holds that
$
  \fwbm_t + \crbm_{s}-\crbm_{t} 
$
and
$
  \fwbm_s
$
are identically distributed ensure that
for all $t \in [0, T]$, $s \in [t, T]$ it holds that
\begin{equation}
\label{a26:eq0}
\begin{split}
  \Exp{| V^{(t)}_s|^2}
&=
  \Exp{|  U(s,\xi+ \fwbm_t + \crbm_{s}-\crbm_{t})|^2} =
  \int_{\R^d}
    \Exp{|  U(s,\xi+ x)|^2} 
  ((\fwbm_t + \crbm_{s}-\crbm_{t} )(\P)_{\mathcal{B}(\R^d)})(dx)  \\
&=
  \int_{\R^d}
    \Exp{|  U(s,\xi+ x)|^2} 
  ((\fwbm_s )(\P)_{\mathcal{B}(\R^d)})(dx)  
=
  \Exp{|  U(s,\xi+ \fwbm_s)|^2}.
\end{split}
\end{equation}
The fact that $V^{(0)}$ is a continuous random field, the fact that $V^{(0)}$ and $\uniform_0$ are independent, Lemma~\ref{contRF_eval_nonneg}, the fact that $\uniform_0$ is uniformly distributed on $[0,T]$, and \eqref{a03} hence establish that
\begin{equation}
\label{a26:eq1}
\begin{split}
  &\left\| 
    [0,T]\times\R^d\times \Omega\ni (t,x,\omega) 
    \mapsto 
    (T-t) \big[U(\uniform_t,x+ \crbm_{\uniform_t}-\crbm _{t})\big](\omega) \in \R
  \right\|_{0}^2 =
  \Exp{
    |T U(\uniform_0, \xi + \crbm_{\uniform_0})|^2
  } \\
  &
=
  T^2 \,
  \EXPP{
    | V^{(0)}_{\uniform_0}|^2
  } 
=
  \frac{T^2}{T}
  \int_0^T
    \EXPP{
      | V^{(0)}_{t}|^2
    } 
  \, dt =
  \frac{T^2}{T}
  \int_0^T
    \Exp{|  U(t,\xi+ \fwbm_t)|^2}
  \, dt
=
   T^2 \|U\|_{1}^2.
\end{split}
\end{equation}
In addition, observe that
the fact that $(V^{(t)})_{t \in [0, T]}$ and $\uniform$ are independent, 
the fact that $V^{(t)}$, $t \in [0, T]$, are continuous random fields,
the fact that for all $t \in [0, T]$ it holds that 
$\uniform_t$ is uniformly distributed on $[t, T]$,
Lemma~\ref{contRF_eval_nonneg},
Tonelli's theorem,
and \eqref{a26:eq0}
demonstrate that
for all $l \in \N$ it holds that
\begin{equation}
\begin{split}
  &\left\|[0,T]\times\R^d\times \Omega\ni (t,x,\omega)\mapsto (T-t) \left[U(\uniform_t,x+\crbm_{\uniform_t}-\crbm_{t})\right](\omega)\in\R\right\|_{l}^2\\
&=
  \frac{1}{T^{l}} 
  \int_0^{T}
    \tfrac{t^{l-1}}{(l-1)!}  \,
    \Exp{
      \left|  (T-t)U(\uniform_t,\fwBM _t+\crbm_{\uniform_t}-\crbm_{t}) \right|^2
    }
  \, dt\\
&=
  \frac{1}{T^{l}} 
  \int_0^{T}
    \tfrac{t^{l-1}}{(l-1)!}  \,
    (T-t)^2 \,
    \EXPP{
      \big|  V^{(t)}_{\uniform_t}\big|^2
    }
  \, dt\\
&=
  \frac{1}{T^{l}} 
  \int_0^{T}
    \tfrac{t^{l-1}}{(l-1)!}  \,
    (T-t)^2 \,
    \tfrac{1}{(T-t)}
    \int_t^{T}
      \EXPP{
        \big|  V^{(t)}_{s}\big|^2
      }
    \, ds
  \, dt\\
&=
  \frac{1}{T^{l}} 
  \int_0^{T}
    \int_0^{T}
      \mathbbm{1}_{ \{  (\mathfrak{t}, \mathfrak{s}) \in [0,T]^2 \colon \mathfrak{t} \leq \mathfrak{s}  \}}(t, s)
      \tfrac{t^{l-1}}{(l-1)!}  \,
      (T-t)\,
        \Exp{|  U(s,\xi+ \fwbm_s)|^2}
    \, dt
  \, ds\\
&\leq
  \frac{T}{T^{l}} 
  \int_0^{T}
    \int_0^{s}
      \tfrac{t^{l-1}}{(l-1)!}  \,
    \, dt \,
      \Exp{|  U(s,\xi+ \fwbm_s)|^2}
  \, ds\\
&=
  \frac{T^2}{T^{l+1}} 
  \int_0^{T}
      \tfrac{s^{l}}{l!}  \,
      \Exp{|  U(s,\xi+ \fwbm_s)|^2}
  \, ds
=
  T^2 \| U\|_{l+1}.
\end{split}
\end{equation}
Combining this and \eqref{a26:eq1} establishes \eqref{a27}.
The proof of  \cref{a26} is thus completed.%
\end{proof}
\begin{lemma}\label{a36}
Assume \cref{a00}, let $k\in\N_0$,  
let
$g \colon \R^d \to \R$ be a $\mathcal{B}(\R^d) /\mathcal{B}(\R)$-measurable function, 
let
$v \colon [0, T] \times \R^d \to \R$ be a $\mathcal{B}([0, T] \times \R^d) /\mathcal{B}(\R)$-measurable function,
let $\crbm\colon[0,T]\times \Omega\to\R^d$ be a standard Brownian motion with continuous sample paths,
and
 assume that $\crbm$ and $\fwbm$ are independent.
Then it holds that
\begin{enumerate}[(i)]
\item \label{a36:item1}

$
\left\|[0,T]\times\R^d\times\Omega\ni (t,x,\omega)\mapsto g(x+\crbm_T(\omega)-\crbm_t(\omega))\in\R\right\|_{k}^2 
    =\tfrac{1}{k!}\E\!\left[\left|g(\fwBM _T)\right|^2\right]\label{a37}
$
and 
\item \label{a36:item2}
$
\left\|v\right\|_{k}\leq \tfrac{1}{\sqrt{k!}}
\left(
  \sup_{t\in[0,T]} \left(\E\!\left[\left|v(t,\fwBM_t)\right|^2\right]\right)^{\! \nicefrac{1}{2}}
\right)
$.
\end{enumerate}
\end{lemma}
\begin{proof}[Proof of \cref{a36}]
First, observe that \cref{a03} and the fact that 
$\crbm_T - \crbm_0 = \crbm_T$ and $\fwbm_T$ are identically distributed
ensure that 
\begin{equation}
 \label{a39}
\begin{split}
&\left\|[0,T]\times\R^d\times\Omega\ni (t,x,\omega)\mapsto g(x+\crbm_T(\omega)-\crbm_t(\omega))\in\R\right\|_{0}^2=\E\!\left[\left|g(\start+\crbm_T - \crbm_0)\right|^2\right]
    =\E\!\left[\left|g(\start+\fwbm_T)\right|^2\right].
\end{split}
\end{equation}%
Next note that the fact that $\fwbm$ and $\crbm$ are independent standard Brownian motions assures that
for all $t \in [0,T]$ the random variables
$\fwbm_T = \fwbm_{t}+ \fwbm_T- \fwbm_t$ and $
\fwbm _{t}+ \crbm_T- \crbm_{t}$ are identically distributed. 
The definition of the semi-norm in \cref{a03} therefore shows that 
for all $l \in \N$ it holds that
\begin{align}
&\left\|[0,T]\times\R^d\times\Omega\ni (t,x,\omega)\mapsto g(x+\crbm_T(\omega)-\crbm_t(\omega))\in\R\right\|_{l}^2
 =
\frac{1}{T^{l}}
      \int_0^{T}\tfrac{t^{l-1}}{(l-1)!}\;
      \E\!\left[\left|g(\fwBM _{t}+\crbm_T-\crbm_{t})\right|^2\right]dt \nonumber
      \\
&=\left[
\frac{1}{T^{l}}
      \int_0^{T}\tfrac{t^{l-1}}{(l-1)!}\,dt\right]\E\!\left[\left|g(\fwBM_T)\right|^2\right]=
  \left[\frac{T^l }{T^{l} l!} \right]\E\!\left[\left|g(\fwBM_T)\right|^2\right]
=
   \frac{\E\!\left[\left|g(\fwBM _T)\right|^2\right]}{l!} .
\end{align}
Combining this and \eqref{a39} proves Item~\eqref{a36:item1}.
Next note that \eqref{a03} implies that
\begin{align}
\label{a36:eq1}
\|v\|_0^2 
=\E\!\left[\left|v(0,\start)\right|^2\right]
=
  \E\!\left[\left| v(0,\fwBM_0)\right|^2\right]
\leq \sup_{t\in[0,T]}\E\!\left[\left|v(t, \start + \fwbm_t)\right|^2\right].
\end{align}
Furthermore, observe that \cref{a03} ensures that 
for all $l \in \N$ it holds that
\begin{align}
\begin{aligned}
\|v\|_{l}^2
&=\frac{1}{T^{l}}\int_0^T \frac{t^{l-1}}{(l-1)!}\;\E\!\left[\left|v(t,\fwBM_t)\right|^2 \right]dt
\leq \left[
\frac{1}{T^{l}}\int_0^T \frac{t^{l-1}}{(l-1)!}\,dt\right]
\sup_{t\in[0,T]}\E\!\left[\left|v(t,\fwBM_t)\right|^2\right] \\
&= 
  \left[\frac{T^l }{T^{l} l!} \right]
  \sup_{t\in[0,T]}\E\!\left[\left|v(t,\fwBM_t)\right|^2\right]
=
\frac{1}{l!}
 \left(\sup_{t\in[0,T]} \E\!\left[\left|v(t,\fwBM_t)\right|^2\right] \right).
\end{aligned}
\end{align}
This and \eqref{a36:eq1} establish Item~\eqref{a36:item2}.
The proof of \cref{a36} is thus completed.
\end{proof}
\section[Multilevel Picard approximations]{Convergence rates for multilevel Picard approximations for semilinear heat equations}
\label{b34}
In this section we develop the overall complexity analysis for the proposed numerical approximation algorithms to establish Theorem~\ref{thm1} in Subsection~\ref{t5} below. More formally, in Subsection~\ref{t1} we formulate the MLP approximation algorithms proposed in this work and the framework which we employ in our error analysis for the proposed MLP approximation algorithms.
In Subsection~\ref{t2} we establish several basic properties of the proposed MLP approximation algorithms and in Subsection~\ref{t3} we prove a priori estimates for the solutions of the PDEs under consideration. Our error analysis for the proposed MLP approximation algorithms can be found in Subsection~\ref{t4}. In Subsection~\ref{t5} we combine this error analysis with a computational cost analysis for the proposed MLP approximation algorithms to accomplish the overall complexity analysis for the proposed MLP approximation algorithms.

%
%
%
\subsection{Setting}\label{t1}
In this subsection we formulate the MLP approximation algorithms and introduce the framework which we employ in our error analysis for the proposed MLP approximation algorithms.

\begin{setting}\label{b35}
Assume \cref{a00}, let
$g \in C( \R^d, \R)$, $u \in C( [0,T]\times \R^d , \R)$ satisfy
for all $t\in[0,T]$, $x\in\R^d$ that
\begin{equation}
\label{b36b}
\begin{split}
  &\E\big[
    |g(x+\fwbm_{t})|
  \big]
+
  \int_0^T 
    \left(\E\!\left[\left|u(s,\xi+\fwbm_s)\right|^2\right]\right)^{\!\nicefrac{1}{2}}
  \,ds \\
&\quad
  +
  \int_t^T 
  \E\big[
    \left|
      (\funcF(u))(s,x+\fwbm_{s-t})
    \right|
    +
    \left|
      (\funcF(0))(s,x+\fwbm_{s-t})
    \right|
  \big] \, ds
  < \infty
\end{split}
\end{equation}
\begin{equation}
\label{b36}
\text{and} \qquad
  u(t,x)
= 
  \Exp{
    g(x+\fwbm_{T-t}) 
    +
    \int_t^{T}
      (\funcF(u))(s,x+\fwbm_{s-t})
    \,ds
  },
\end{equation}
let
$
  \Theta = \cup_{ n \in \N } \Z^n
$,
let
$
  W^{ \theta } \colon [0,T] \times \Omega \to \R^d 
$,
$ \theta \in \Theta $,
be independent standard Brownian motions
with continuous sample paths,
let $\unif^\theta\colon \Omega\to[0,1]$, $\theta\in \Theta$, be independent $\mathcal{U}_{[0,1]}$-distributed random variables, 
assume that $(W^\theta)_{\theta \in \Theta}$, $(\unif^\theta)_{\theta \in \Theta}$, and $\fwbm$ are independent,
let $\uniform^\theta\colon [0,T]\times \Omega\to [0, T]$, $\theta\in\Theta$, satisfy 
for all $t\in [0,T]$, $\theta\in \Theta$ that 
$\uniform^\theta _t = t+ (T-t)\unif^\theta$,
and
let
$ 
  { U}_{ n,M}^{\theta } \colon [0, T] \times \R^d \times \Omega \to \R
$, $n,M\in\Z$, $\theta\in\Theta$,
satisfy
for all $n,M \in \N$, $\theta\in\Theta $, 
$ t \in [0,T]$, $x\in\R^d $
that $U_{-1,M}^{\theta}(t,x)=
U_{0,M}^{\theta}(t,x)=0$ and 
\begin{equation}
\begin{split}\label{b37}
  &U_{n,M}^{\theta}(t,x)
=
  \frac{1}{M^n}
  \left[ \sum_{i=1}^{M^n} 
      g\big(x+W^{(\theta,0,-i)}_T-W^{(\theta,0,-i)}_t\big)
  \right] \\
&  +
  \sum_{l=0}^{n-1} \frac{(T-t)}{M^{n-l}}
    \left[\sum_{i=1}^{M^{n-l}}
      \big(\funcF(U_{l,M}^{(\theta,l,i)})-\1_{\N}(l)\funcF( U_{l-1,M}^{(\theta,-l,i)})\big)
      \big(\uniform_t^{(\theta,l,i)},x+W_{\uniform_t^{(\theta,l,i)}}^{(\theta,l,i)}-W_t^{(\theta,l,i)}\big)
    \right].
\end{split}
\end{equation}
\end{setting}
\subsection{Properties of the approximations}\label{t2}

In this subsection we establish several basic properties of the in Subsection~\ref{t1} introduced MLP approximation algorithms.

\begin{lemma}
\label{properties_approx}
Assume \cref{b35}. 
Then
\begin{enumerate}[(i)]
\item  \label{properties_approx:item1}
it holds 
for all $n \in \N_0$, $M \in \N$, $\theta\in\Theta $ that
$
  { U}_{ n,M}^{\theta } \colon [0, T] \times \R^d \times \Omega \to \R
$
is a continuous random field,

\item  \label{properties_approx:item2}
it holds for all $n \in \N_0$, $M \in \N$, $\theta \in \Theta$ that
$
  \sigma_\Omega( U^\theta_{n, M} )
\subseteq
  \sigma_\Omega( (\unif^{(\theta, \vartheta)})_{\vartheta \in \Theta}, (W^{(\theta, \vartheta)})_{\vartheta \in \Theta}) )
$,

\item  \label{properties_approx:item3}
it holds 
for all $n \in \N_0$, $M \in \N$, $\theta\in\Theta $ that
${ U}_{ n,M}^{\theta }$, $W^\theta$, and $\unif^\theta$ are independent,

\item  \label{properties_approx:item4}
it holds
for all $n, m \in \N_0$, $M \in \N$, $i,j,k,l, \in \Z$, $\theta \in \Theta$ with $(i,j) \neq (k,l)$ 
that
$
  U^{(\theta,i,j)}_{n,M}
$
and
$
  U^{(\theta,k,l)}_{m,M}
$
are independent,
and
\item  \label{properties_approx:item5}
it holds 
for all $n \in \N_0$, $M \in \N$ that 
$
  (U^\theta_{n, M})_{\theta \in \Theta}
$
are identically distributed.
\end{enumerate}
\end{lemma}

\begin{proof}[Proof of Lemma~\ref{properties_approx}]
First, observe that the hypothesis that 
for all $M \in \N$, $\theta\in\Theta $ it holds that
$U^\theta_{0, M} = 0$,
\eqref{b37},
Item~\eqref{a14:item1} in Lemma~\ref{a14},
the fact 
for all $\theta \in \Theta$ it holds that $W^\theta$ and $\uniform^\theta$ are continuous random fields,
the hypothesis that $g$ is continuous, and induction on $\N_0$ establish Item~\eqref{properties_approx:item1}.
Next note that Item~\eqref{a14:item1} in Lemma~\ref{a14}, Beck et al.~\cite[Lemma 2.4]{Becketal2018}, and Item~\eqref{properties_approx:item1} assure that 
for all $n \in \N_0$, $M \in \N$, $\theta\in\Theta $ it holds that 
$F(U^\theta_{n, M})$ is $ ( \mathcal{B}([0, T] \times \R^d ) \otimes \sigma_\Omega(U^\theta_{n, M}) )/ \mathcal{B}(\R )$-measurable.
The hypothesis that 
for all $M \in \N$, $\theta\in\Theta $ it holds that
$U^\theta_{0, M} = 0$,
\eqref{b37}, 
the fact that 
for all $\theta \in \Theta$ it holds that 
$W^\theta$ is $ ( \mathcal{B}([0, T]) \otimes \sigma_\Omega(W^\theta) )/ \mathcal{B}(\R^d )$-measurable,
the fact that 
for all $\theta \in \Theta$ it holds that 
$\uniform^\theta$ is $ ( \mathcal{B}([0, T]) \otimes \sigma_\Omega(\unif^\theta) )/ \mathcal{B}([0, T] )$-measurable,
and induction on $\N_0$ prove Item~\eqref{properties_approx:item2}.
Furthermore, observe that Item~\eqref{properties_approx:item2} and the fact that 
for all $\theta \in \Theta$ it holds that
$(\unif^{(\theta, \vartheta)})_{\vartheta \in \Theta}, (W^{(\theta, \vartheta)})_{\vartheta \in \Theta}$,
$W^\theta$, and $\unif^\theta$ are independent establish Item~\eqref{properties_approx:item3}.
In addition, note that Item~\eqref{properties_approx:item2} and the fact that
for all $i,j,k,l, \in \Z$, $\theta \in \Theta$ with $(i,j) \neq (k,l)$ it holds that 
$((\unif^{(\theta,i,j, \vartheta)})_{\vartheta \in \Theta}, (W^{(\theta,i,j, \vartheta)})_{\vartheta \in \Theta})$
and
$((\unif^{(\theta,k,l, \vartheta)})_{\vartheta \in \Theta}, (W^{(\theta,k,l, \vartheta)})_{\vartheta \in \Theta})$
are independent prove Item~\eqref{properties_approx:item4}.
Finally, observe that the hypothesis that 
for all $M \in \N$, $\theta\in\Theta $ it holds that
$U^\theta_{0, M} = 0$,
the hypothesis that $(W^\theta)_{\theta \in \Theta}$ are i.i.d., 
the hypothesis that $(\uniform^\theta)_{\theta \in \Theta}$ are i.i.d., 
Items~\eqref{properties_approx:item1}--\eqref{properties_approx:item4}, Corollary~\ref{id_RF}, and induction on $\N_0$ establish Item~\eqref{properties_approx:item5}.
The proof of Lemma~\ref{properties_approx} is thus completed.
\end{proof}
\begin{lemma}[Approximations are integrable]\label{b38}
Assume \cref{b35}.
Then 
it holds
for all
$n\in\N_0$, $M\in\N$, $\theta\in\Theta$, $t\in [0,T]$, $s\in[t,T]$, $x\in \R^d$ that
\begin{equation} 
\begin{split} \label{b39}
  \E\!\left[\left| U_{n,M}^{\theta}(s,x+W_{s-t}^\theta)
\right|
  +\int_t^T \left| U_{n,M}^{\theta}(r,x+W_{r-t}^\theta)
\right|+
    \left|(\funcF( U_{n,M}^{\theta}))(r,x+W_{r-t}^\theta)
\right|\,dr
\right]<\infty.
\end{split}     
\end{equation}
\end{lemma}
\begin{proof}[Proof of Lemma~\ref{b38}]
Throughout this proof let $M \in \N$, $\theta \in \Theta$, $x \in \R^d$.
We claim that 
for all $n\in\N_0$, $t\in [0,T]$, $s\in[t,T]$ it holds that
\begin{equation}
\begin{split} \label{b38:eq1}
  \E\!\left[
    \left| U_{n,M}^{\theta}(s,x+W_{s-t}^\theta) \right|
    +
    \int_t^T \left| U_{n,M}^{\theta}(r,x+W_{r-t}^\theta)\right|
    +
    \left| (\funcF( U_{n,M}^{\theta}))(r,x+W_{r-t}^\theta) \right|\,dr
  \right]
<
  \infty.
\end{split}     
\end{equation}
We now prove \eqref{b38:eq1} by induction on $n \in \N_0$.
For the base case $n=0$, note that \cref{b36b} and the fact that $U^\theta_{0,M} = 0$ ensure that 
for all $t\in[0,T]$, $s \in [t, T]$ it holds that
\begin{equation} 
\begin{split}
  &\E\!\left[
    \left| U_{0,M}^{\theta}(s,x+W_{s-t}^\theta) \right|
    +
    \int_t^T \left| U_{0,M}^{\theta}(r,x+W_{r-t}^\theta)\right|
    +
    \left| (\funcF( U_{0,M}^{\theta}))(r,x+W_{r-t}^\theta) \right|\,dr
  \right]\\
&=
  \E\!\left[
    \int_t^T
      \left|  (\funcF(0))(r,x+W^\theta_{r-t})  \right|
    \,dr
  \right]
<
  \infty.
\end{split}
\end{equation}
This establishes \eqref{b38:eq1} in the base case $n=0$. 
For the induction step $\N_0 \ni n-1 \rightarrow n \in \N$ let $n\in \N$ and assume that 
for all $k \in \N_0 \cap [0, n)$, $t\in[0,T]$, $s \in [t, T]$ it holds that 
\begin{equation}
\begin{split} \label{b38:eq2}
  \E\!\left[
    \left| U_{k,M}^{\theta}(s,x+W_{s-t}^\theta) \right|
    +
    \int_t^T \left| U_{k,M}^{\theta}(r,x+W_{r-t}^\theta)\right|
    +
    \left| (\funcF( U_{k,M}^{\theta}))(r,x+W_{r-t}^\theta) \right|\,dr
  \right]
<
  \infty.
\end{split}
\end{equation}
Observe that the triangle inequality and \eqref{b37} ensure that 
for all $t\in[0,T]$, $s \in [t, T]$ it holds that 
\begin{equation}
\label{b38:eq3}
\begin{split}
  &\E\!\left[ \left| 
    U_{n,M}^{\theta}(s,x+W_{s-t}^\theta) 
  \right| \right] 
\leq 
  \frac{1}{M^{n}}\sum_{i=1}^{M^n}
  \E\!\left[\left| 
  	g(x+W^\theta_{s-t}+W^{(\theta,0,-i)}_T-W^{(\theta,0,-i)}_s)\right|
  \right] \\
&\quad
  +
  \sum_{l=0}^{n-1}\tfrac{(T-s)}{M^{n-l}}
  \sum_{i=1}^{M^{n-l}}
    \E\!\left[\left|\Big(\funcF(U_{l,M}^{(\theta,l,i)})-\1_{\N}(l)\funcF( U_{l-1,M}^{(\theta,-l,i)})\Big)
    (\uniform_s^{(\theta,l,i)},x+W^\theta_{s-t}+W_{\uniform_s^{(\theta,l,i)}}^{(\theta,l,i)}-W_s^{(\theta,l,i)})\right|\right].
  \\
\end{split}
\end{equation}
In addition, note that the fact that 
for all $i \in \Z$ it holds that $W^\theta$ and $W^{(\theta,0,i)}$ are independent Brownian motions assures that
for all $t\in[0,T]$, $s \in [t, T]$, $i \in \Z$ it holds that 
\begin{equation}
\label{b38:eq4}
  \E\!\left[\left| 
    g(x+W^\theta_{s-t}+W^{(\theta,0,i)}_T-W^{(\theta,0,i)}_s)
  \right| \right]
=
  \E\!\left[\left| 
    g(x+W^\theta_{(s-t) + (T-s)})
  \right| \right]
=
    \E\!\left[\left| 
    g(x+W^\theta_{T-t})
  \right| \right].
\end{equation}
Moreover, note that 
Lemma~\ref{properties_approx},
the hypothesis that $(W^\theta)_{\theta \in \Theta}$ are i.i.d., 
the hypothesis that $(\uniform^\theta)_{\theta \in \Theta}$ are i.i.d., 
the hypothesis that $(W^\theta)_{\theta \in \Theta}$ and $(\uniform^\theta)_{\theta \in \Theta}$ are independent, 
Lemma~\ref{contRF_eval_nonneg}, and the triangle inequality
assure that
for all $t\in[0,T]$, $s \in [t, T]$ it holds that 
\begin{equation}
\label{b38:eq5}
\begin{split}
  &\sum_{l=0}^{n-1}\tfrac{(T-s)}{M^{n-l}}
  \sum_{i=1}^{M^{n-l}}
    \E\!\left[\left|\Big(\funcF(U_{l,M}^{(\theta,l,i)})-\1_{\N}(l)\funcF( U_{l-1,M}^{(\theta,-l,i)})\Big)
    (\uniform_s^{(\theta,l,i)},x+W^\theta_{s-t}+W_{\uniform_s^{(\theta,l,i)}}^{(\theta,l,i)}-W_s^{(\theta,l,i)})\right|\right] \\
&=
  \sum_{l=0}^{n-1}(T-s)
    \E\!\left[\left|\Big(\funcF(U_{l,M}^{(\theta,l,0)})-\1_{\N}(l)\funcF( U_{l-1,M}^{(\theta,-l,0)})\Big)
    (\uniform_s^{(\theta,l,0)},x+W^\theta_{s-t} +W_{\uniform_s^{(\theta,l,0)}}^{(\theta,l,0)}-W_s^{(\theta,l,0)})\right|\right] \\
&\leq
  2
  \sum_{l=0}^{n-1}
    (T-s)\E\!\left[\left|\big(\funcF(U_{l,M}^{(\theta,l,0)})\big)
    (\uniform_s^{(\theta,l,0)},x+W^\theta_{s-t}+W_{\uniform_s^{(\theta,l,0)}}^{(\theta,l,0)}-W_s^{(\theta,l,0)})\right|\right].
\end{split}
\end{equation}
Furthermore, observe that Lemma~\ref{properties_approx}, 
the fact that 
for all $l \in \Z $ it holds that
$W^\theta$, $W^{(\theta,l,0)}$, $\uniform^{(\theta, l, 0)}$, and $U^{(\theta, l, 0)}$ are independent, and
Lemma~\ref{contRF_eval_nonneg}
demonstrate that
for all $t\in[0,T]$, $s \in [t, T]$, $l \in \N_0 \cap [0, n)$ it holds that 
\begin{equation}
\begin{split}
  &(T-s)
  \E\!\left[\left|\big(\funcF(U_{l,M}^{(\theta,l,0)})\big)
    (\uniform_s^{(\theta,l,0)},x+W^\theta_{s-t}+W_{\uniform_s^{(\theta,l,0)}}^{(\theta,l,0)}-W_s^{(\theta,l,0)})\right|\right] \\
&=
  \int_s^T
    \E\!\left[\big|\big(\funcF(U_{l,M}^{(\theta,l,0)})\big)
      (r,x+W^\theta_{s-t}+W_{r}^{(\theta,l,0)}-W_s^{(\theta,l,0)})\big|\right]
  \, dr \\
&=
  \int_s^T
    \E\!\left[\big|\big(\funcF(U_{l,M}^{(\theta,l,0)})\big)
      (r,x + W_{(s-t) + (r-s)}^{(\theta,l,0)})\big|\right]
  \, dr 
=
  \int_s^T
    \E\!\left[\left|\big(\funcF(U_{l,M}^{\theta})\big)
      (r,x + W_{r-t}^{\theta})\right|\right]
  \, dr.
\end{split}
\end{equation}
Combining this, \eqref{b38:eq3}, \eqref{b38:eq4}, and  \eqref{b38:eq5} with  \eqref{b36b}, \eqref{b38:eq2}, and Tonelli's theorem establishes that
for all $t\in[0,T]$, $s \in [t, T]$ it holds that 
\begin{equation}
\label{b38:eq6}
\begin{split}
  \E\!\left[ \left| 
    U_{n,M}^{\theta}(s,x+W_{s-t}^\theta) 
  \right| \right] &\leq 
  \frac{1}{M^{n}}\sum_{i=1}^{M^n}
  \E\!\left[\left| 
  	g(x+W^\theta_{T-t})\right|
  \right]
  +
  2\sum_{l=0}^{n-1}
    \int_t^T
      \E\!\left[\left|\big(\funcF(U_{l,M}^{\theta})\big)
        (r,x + W_{r-t}^{\theta})\right|\right]
    \, dr \\
&=
  \E\!\left[\left| 
  	g(x+W^\theta_{T-t})\right|
  \right]
  +
  2\sum_{l=0}^{n-1}
    \E\!\left[
      \int_t^T\left|\big(\funcF(U_{l,M}^{\theta})\big)
        (r,x + W_{t-r}^{\theta})\right|
        \, dr 
     \right]    
<
  \infty.
\end{split}
\end{equation}
This, Tonelli's theorem, and \eqref{b38:eq2} imply that 
for all $t\in[0,T]$ it holds that 
\begin{equation}
\label{b38:eq7}
\begin{split}
  &\Exp{ \int_t^T \left| U_{n,M}^{\theta}(s,x+W_{s-t}^\theta)\right| \, ds} 
= 
  \int_t^T \Exp{ \left| U_{n,M}^{\theta}(s,x+W_{s-t}^\theta)\right| } \, ds \\
&\leq
  (T-t) 
  \left[
    \E\!\left[\left| 
  	  g(x+W^\theta_{T-t})\right|
    \right]
    +
    2\sum_{l=0}^{n-1}
      \int_t^T
        \E\!\left[\left|\big(\funcF(U_{l,M}^{\theta})\big)
          (r,x + W_{r-t}^{\theta})\right|\right]
      \, dr 
  \right]
<
  \infty.
\end{split}
\end{equation}
The triangle inequality, Tonelli's theorem, \eqref{a01}, and \cref{b36b} hence prove that
for all $t\in[0,T]$ it holds that
\begin{align}
\begin{aligned}
&  \E\left[\int_t^T\big|(F(U_{n,M}^{\theta}))(s,x+W_{s-t}^\theta)\big|\,ds \right]
=
\int_t^T  \E\Big[\big|(F(U_{n,M}^{\theta}))(s,x+W_{s-t}^\theta)\big|\Big]\,ds
\\
&\leq 
\int_t^T  \E\Big[\left|\big(F(U_{n,M}^{\theta})-F(0)\big)(s,x+W_{s-t}^\theta)\right|\Big]\,ds
+\int_t^T  \E\Big[\big|(F(0))(s,x+W_{s-t}^\theta)\big|\Big]\,ds
\\
&\leq 
\int_t^T  \E\Big[L\big| U_{n,M}^{\theta}(s,x+W_{s-t}^\theta)\big|\Big]\,ds
+\int_t^T  \E\Big[\big|(F(0))(s,x+W_{s-t}^\theta)\big|\Big]\,ds<\infty.
\end{aligned}
\end{align}
Induction, \eqref{b38:eq6}, and \eqref{b38:eq7} hence establish \eqref{b38:eq1}. 
The proof of  \cref{b38} is thus completed.
\end{proof}

\subsection{Upper bound for the exact solution}\label{t3}
In this subsection we establish the upper bound~\eqref{upper_exact:concl1} below for the
exact solution which is well-known in the literature and included here for the
reason of being self-contained.

\begin{lemma}[Upper bound for exact solution]
\label{upper_exact}
Assume \cref{b35}.
Then it holds that
\begin{equation}  
\begin{split}\label{upper_exact:concl1}
  \sup_{t\in[0,T]}\left(\E\!\left[\left|u(t,\start +\fwbm_t)\right|^2\right]\right)^{\!\nicefrac{1}{2}}
  \leq e^{LT} \! \left[
    \left(\E\!\left[\left|g(\xi+\fwbm_T)\right|^2\right]\right)^{\!\nicefrac{1}{2}}
    +T\|F(0)\|_1\right].
\end{split}     
\end{equation}
\end{lemma}

\begin{proof}[Proof of Lemma~\ref{upper_exact}]
Throughout this proof let 
$\crbm \colon [0, T] \times \Omega \to \R^d$ be a standard Brownian motion with continuous sample paths, 
assume that $\fwbm$ and $\crbm$ are independent, 
let $\mu_{t} \colon \mathcal{B}(\R^d) \to [0,1]$, $t \in [0,T]$, be the probability measures which satisfy 
for all $t \in [0,T]$, $B \in \mathcal{B}(\R^d)$ that
$
  \mu_t(B) 
=
  \P( \xi +  \crbm_t \in B )
  $,
and assume w.lo.g.\ that 
$
  \E\!\left[\left|g(\xi+\fwbm_T)\right|^2\right]
  +
  \|F(0)\|_1
<
  \infty.
$
Observe that the integral transformation theorem, \eqref{b36}, and the triangle inequality assure that
for all $t \in [0,T]$ it holds that 
\begin{equation}
\label{upper_exact:eq1}
\begin{split}
  &\left( \Exp{ | u(t,\start +\fwbm_t)|^2 } \right)^{ \!\nicefrac{1}{2}}
=
  \left( \Exp{ | u(t,\start +\crbm_t)|^2 } \right)^{ \!\nicefrac{1}{2}}
=
  \left( 
    \int_{\R^d}
      | u(t, x)|^2
    \, \mu_t(dx)
  \right)^{ \! \nicefrac{1}{2}} \\
&=
  \left( 
    \int_{\R^d}
      \left| 
        \Exp{
          g(x+\fwbm_{T-t}) 
          +
          \int_t^{T}
            (\funcF(u))(s,x+\fwbm_{s-t})
          \,ds
        }
      \right|^2
    \, \mu_t(dx)
  \right)^{ \! \! \nicefrac{1}{2}} \\
&\leq
  \left( 
    \int_{\R^d}
      \left| 
        \EXPP{
          g(x+\fwbm_{T-t}) 
        }
      \right|^2
    \, \mu_t(dx)
  \right)^{ \! \! \nicefrac{1}{2}}
  +
  \left( 
    \int_{\R^d}
      \left| 
        \Exp{
          \int_{t}^T
            (\funcF(u))(s,x+\fwbm_{s-t})
          \, ds
        }
      \right|^2
    \, \mu_t(dx)
  \right)^{ \! \! \nicefrac{1}{2}}.
\end{split}
\end{equation}
Jensen's inequality hence assures that
for all $t \in [0,T]$ it holds that
\begin{multline}
\label{upper_exact:eq2}
  \left( \Exp{ | u(t,\start +\fwbm_t)|^2 } \right)^{ \! \! \nicefrac{1}{2}} 
\leq
  \left( 
    \int_{\R^d}
        \Exp{
          |  g(x+\fwbm_{T-t})  |^2
        }
    \, \mu_t(dx)
  \right)^{ \! \nicefrac{1}{2}} \\
  +
  \left( 
    \int_{\R^d}
        \Exp{
          \left( \int_{t}^T
            \big|  (\funcF(u))(s,x+\fwbm_{s-t})  \big|
          \, ds \right)^2
        }
    \, \mu_t(dx)
  \right)^{ \!  \!\nicefrac{1}{2}}.
\end{multline}
Furthermore, observe that Lemma~\ref{contRF_eval_nonneg}
and the fact that $\fwbm$ and $\crbm$ are independent Brownian motions demonstrate that 
for all $t \in [0,T]$ it holds that
\begin{equation}
\label{upper_exact:eq3}
\begin{split}
  \left( 
    \int_{\R^d}
        \Exp{
          |  g(x+\fwbm_{T-t})  |^2
        }
    \, \mu_t(dx)
  \right)^{ \! \nicefrac{1}{2}}
&=
 \left( 
        \Exp{
          |  g( \start + \crbm_t+\fwbm_{T-t})  |^2
        }
  \right)^{ \! \nicefrac{1}{2}} =
  \left( 
        \Exp{
          |  g ( \start + \fwbm_{T} )  |^2
        }
  \right)^{ \! \nicefrac{1}{2}}.
\end{split}
\end{equation}
In addition, note that Minkowski's integral inequality, Lemma~\ref{contRF_eval_nonneg} 
and the fact that $\fwbm$ and $\crbm$ are independent Brownian motions imply that
for all $t \in [0,T]$ it holds that
\begin{equation}
\label{upper_exact:eq4}
\begin{split}
  &\left( 
    \int_{\R^d}
        \Exp{
          \left( \int_{t}^T
            \big|  (\funcF(u))(s,x+\fwbm_{s-t})  \big|
          \, ds \right)^2
        }
    \, \mu_t(dx)
  \right)^{ \! \nicefrac{1}{2}} \leq
  \int_{t}^T
     \left( 
      \int_{\R^d}
          \Exp{
              |  (\funcF(u))(s,x+\fwbm_{s-t})  |^2
          }
      \, \mu_t(dx)
    \right)^{ \! \nicefrac{1}{2}} 
  ds \\
&=
  \int_{t}^T
     \left( 
          \Exp{
              |  (\funcF(u))(s, \start + \crbm_t+\fwbm_{s-t})  |^2
          }
    \right)^{ \! \nicefrac{1}{2}} 
  ds 
=
   \int_{t}^T
     \left( 
          \Exp{
              |  (\funcF(u))(s, \start + \fwbm_{s})  |^2
          }
    \right)^{ \! \nicefrac{1}{2}} 
  ds.
\end{split}
\end{equation}
This, the triangle inequality, and \eqref{a01} assure that
for all $t \in [0,T]$ it holds that
\begin{equation}
\label{upper_exact:eq6}
\begin{split}
  &\left( 
    \int_{\R^d}
        \Exp{
          \left( \int_{t}^T
            \big|  (\funcF(u))(s,x+\fwbm_{s-t})  \big|
          \, ds \right)^2
        }
    \, \mu_t(dx)
  \right)^{ \! \nicefrac{1}{2}} \\
&\leq
   \int_{t}^T
     \left( 
          \Exp{
              |  (\funcF(0))(s, \start+\fwbm_{s})  |^2
          }
    \right)^{ \! \nicefrac{1}{2}} 
  \, ds
  + 
  \int_{t}^T
     \left( 
          \Exp{
              | (\funcF(u) - \funcF(0))(s,\start+\fwbm_{s})  |^2
          }
    \right)^{ \! \nicefrac{1}{2}} 
  \, ds \\
&\leq
   \int_{t}^T
     \left( 
          \Exp{
              |  (\funcF(0))(s,\start+\fwbm_{s})  |^2
          }
    \right)^{ \! \nicefrac{1}{2}} 
  \, ds
  + 
  \int_{t}^T
     \left( 
          \Exp{
              L^2  | u(s,\start+\fwbm_{s})  |^2
          }
    \right)^{ \! \nicefrac{1}{2}} 
  \, ds.
\end{split}
\end{equation}
Furthermore, note that Jensen's inequality and \eqref{a03} ensure that
for all $t \in [0, T)$ it holds that
\begin{equation}
\label{upper_exact:eq7}
\begin{split}
  \int_{t}^T 
     \left( 
          \Exp{
              |  (\funcF(0))(s,\start+\fwbm_{s})  |^2
          }
    \right)^{ \! \nicefrac{1}{2}} 
  \, ds 
&=
  (T-t)
  \left(
    \tfrac{1}{(T-t)}
    \int_{t}^T
       \left( 
          \Exp{
              |  (\funcF(0))(s,\start+\fwbm_{s})  |^2
          }
      \right)^{ \! \nicefrac{1}{2}} 
    \, ds
  \right)  \\
&\leq
  (T-t)
  \left(
    \tfrac{1}{(T-t)}
    \int_{t}^T
          \Exp{
              |  (\funcF(0))(s,\start+\fwbm_{s})  |^2
          }
    \, ds
  \right)^{ \! \nicefrac{1}{2}}  \\
&\leq
  \sqrt{T}
  \left(
    \int_{0}^T
          \Exp{
              |  (\funcF(0))(s,\start+\fwbm_{s})  |^2
          }
    \, ds
  \right)^{ \! \nicefrac{1}{2}}
=
  T \| F(0)\|_1.
\end{split}
\end{equation}
Combining this with \eqref{upper_exact:eq2}, \eqref{upper_exact:eq3}, and \eqref{upper_exact:eq6} implies that
for all $t \in [0,T]$ it holds that
\begin{equation}
\label{upper_exact:eq8}
\begin{split}
  &\left( \Exp{ | u(t,\start +\fwbm_t)|^2 } \right)^{ \!\nicefrac{1}{2}} 
\leq
   \left( 
        \Exp{
          |  g ( \start + \fwbm_{T} )  |^2
        }
  \right)^{ \! \nicefrac{1}{2}}
  +
  T \| F(0)\|_1
  +
  L
  \int_{t}^T
     \left( 
          \Exp{
              | u(s,\start+\fwbm_{s})  |^2
          }
    \right)^{ \! \nicefrac{1}{2}} 
  \, ds.
\end{split}
\end{equation}
The hypothesis that 
$
  \int_{0}^T
     \left( 
          \Exp{
              | u(t,\start+\fwbm_{t})  |^2
          }
    \right)^{ \! \nicefrac{1}{2}} 
  \, dt
< 
  \infty
$
and Gronwall's integral inequality hence establish that
for all $t \in [0, T]$ it holds that
\begin{equation}
\begin{split}
  \left(
    \Exp{
      | u(t,\start+\fwbm_{t})  |^2
    }
  \right)^{\nicefrac{1}{2}}
&\leq
  e^{L(T-t)}\left[
    \left(\E\!\left[\left|g(\xi+\fwbm_T)\right|^2\right]\right)^{\!\nicefrac{1}{2}}
    +T\|F(0)\|_1\right] \leq
  e^{LT}\left[
    \left(\E\!\left[\left|g(\xi+\fwbm_T)\right|^2\right]\right)^{\!\nicefrac{1}{2}}
    +T\|F(0)\|_1\right].
\end{split}
\end{equation}
The proof of Lemma~\ref{upper_exact} is thus completed.
\end{proof}

\subsection{Error analysis for multilevel Picard approximations}\label{t4}
In this subsection we provide in Theorem~\ref{b45} below our error analysis for the MLP approximation algorithms introduced in Subsection~\ref{t1}.

\begin{theorem}\label{b45}
Assume \cref{b35} and
let $N,M\in \N$.
Then it holds that
\begin{equation}  
\label{eq:main_estimate}
\begin{split}
  &\left(\E\!\left[\left|U^0_{N,M}(0,\start)-u(0,\start)\right|^2\right]\right)^{\!\!\nicefrac{1}{2}}
\leq
  e^{LT}
  \left[\left(\E\!\left[\left|g(\start + \fwbm_T)\right|^2\right]\right)^{\!\!\nicefrac{1}{2}}+T\left\|\funcF(0)\right\|_1  \right]
  \frac{e^{M/2}(1+2LT)^{N}}{M^{N/2}}
.
\end{split}     
\end{equation}
\end{theorem}
\begin{proof}[Proof of Theorem~\ref{b45}]
Throughout this proof assume w.l.o.g.\ that 
$\E\!\left[\left|g(\start + \fwbm_T)\right|^2\right]+\left\|\funcF(0)\right\|_1 < \infty$. 
Note that Item~\eqref{a07a:item1} in Lemma~\ref{a07a} and Lemma~\ref{b38} assure that  
for all $n \in \N$, $k \in \N_0$ it holds that
\begin{equation}
\label{b45:eq01}
  \left\|U_{n,M}^0-u\right\|_{k}
\leq 
  \left\|U_{n,M}^0-\E \!\left[U_{n,M}^0\right]\right\|_{k}+\left\|\E\!\left[U_{n,M}^0\right]-u\right\|_{k}.
\end{equation}
Next observe that Lemma~\ref{contRF_eval_nonneg}, Item~\eqref{a14:item1} in Lemma~\ref{a14}, Lemma~\ref{properties_approx},  Lemma~\ref{b38}, Corollary~\ref{id_RF}, 
and the fact that  
for all $\theta \in \Theta$, $t \in [0, T]$ it holds that
$ \uniform_t^\theta$
is uniformly distributed on $[t, T]$
assure that
for all $t\in[0,T]$, $x\in \R^d$, $n\in\N$, $i,j,k \in \Z$, $\theta \in \Theta$ it holds that
\begin{equation}
\label{b45:eq02}
\begin{split}
  (T-t)
  \Exp{
    \big|(\funcF(U_{n,M}^{(\theta, k, i)}))(\uniform_t^{(\theta, j, i)},x+W^{(\theta, j, i)}_{\uniform_t^{(\theta, j, i)}}-W_t^{(\theta, j, i)}\big|
  } 
&=
  \int_t^T 
    \Exp{
      \big|(\funcF(U_{n,M}^{(\theta, k, i)}))(s,x+W_{s}^{(\theta, j, i)}-W_t^{(\theta, j, i)})\big|
    }
  ds \\
&=
  \int_t^T 
    \Exp{
      \big|(\funcF(U_{n,M}^\theta))(s,x+W_{s-t}^\theta)\big|
    }
  ds 
< 
  \infty.
\end{split}
\end{equation}
Combining this with the fact that 
for all $t \in [0, T]$, $\theta \in \Theta$ it holds that
$
  \Exp{
    \left|g(x+W_T^\theta-W_t^\theta)\right|
  }
=
  \Exp{
    \left|g(x+W_{T-t}^\theta)\right|
  }
< 
  \infty
$ 
and \eqref{b37} ensures that
for all $n \in \N$, $t \in [0, T]$, $x \in \R^d$ it holds that
\begin{equation}
\label{b45:eq03}
\begin{split}
  &\E \!\left[U_{n,M}^0(t, x)\right]
=
  \frac{1}{M^n}
  \left[ \sum_{i=1}^{M^n} 
    \Exp{ g\big(x+W^{(0,0,-i)}_T-W^{(0,0,-i)}_t\big)} 
  \right]\\
&\quad
  +
  \sum_{l=0}^{n-1} \frac{(T-t)}{M^{n-l}}
    \left[\sum_{i=1}^{M^{n-l}}
      \Exp{
        \left(\funcF\big(U_{l,M}^{(0,l,i)}\big)-\1_{\N}(l)\funcF\big( U_{l-1,M}^{(0,-l,i)}\big)\right)
        \left(\uniform_t^{(0,l,i)},x+W_{\uniform_t^{(0,l,i)}}^{(0,l,i)}-W_t^{(0,l,i)}\right)
      }
    \right].
\end{split}
\end{equation}
This and \eqref{b37} imply that 
for all $n \in \N$, $k \in \N_0$ it holds that
\begin{equation}
\label{b45:eq04}
\begin{split}
  &\left\|U_{n,M}^0-\E \!\left[U_{n,M}^0\right]\right\|_{k} \\
&=
  \left\| 
    [0,T]\times \R^d\times \Omega\ni (t,x,\omega) \mapsto 
    U_{n,M}^0 (t, x, \omega) - \E \!\left[U_{n,M}^0(t, x)\right]
  \right\|_k \\
&\leq
  \Bigg\| 
    [0,T]\times \R^d\times \Omega\ni (t,x,\omega) \\
     &\quad\quad \mapsto   
     \Bigg[ \sum_{i=1}^{M^n} 
       \tfrac{1}{M^n} 
       \left(
         \left[ g\big(x+W^{(0,0,-i)}_T-W^{(0,0,-i)}_t\big) \right] \!(\omega)
         -
         \Exp{ g\big(x+W^{(0,0,-i)}_T-W^{(0,0,-i)}_t\big)} 
       \right)
     \Bigg] \in \R
  \Bigg\|_k\\
&\quad+
  \sum_{l=0}^{n-1}
    \Bigg\|
      [0,T]\times \R^d\times \Omega\ni (t,x,\omega)\\
&\quad\quad\quad\quad
      \mapsto 
      \Bigg[ \sum_{i=1}^{M^{n-l}}  
        \tfrac{T-t}{M^{n-l}}
        \Bigg(
          \left[
            \big(\funcF(U_{l,M}^{(0,l,i)})-\1_{\N}(l)\funcF( U_{l-1,M}^{(0,-l,i)})\big)
            \big(\uniform_t^{(0,l,i)},x+W_{\uniform_t^{(0,l,i)}}^{(0,l,i)}-W_t^{(0,l,i)}\big)
          \right] \! (\omega)\\
&\qquad\qquad\quad\quad\quad
          - 
          \EXPPP{
            \big(\funcF(U_{l,M}^{(0,l,i)})-\1_{\N}(l)\funcF( U_{l-1,M}^{(0,-l,i)})\big)
            \big(\uniform_t^{(0,l,i)},x+W_{\uniform_t^{(0,l,i)}}^{(0,l,i)}-W_t^{(0,l,i)}\big)
          }
        \Bigg)
      \Bigg] \in\R
    \Bigg\|_k.
\end{split}
\end{equation}
Moreover, note that Lemma~\ref{properties_approx},
the hypothesis that $(W^\theta)_{\theta \in \Theta}$ are i.i.d., 
the hypothesis that $(\uniform^\theta)_{\theta \in \Theta}$ are i.i.d., Item~\eqref{a07a:item1} in Lemma~\ref{a07a}, and Corollary~\ref{id_RF}
ensure that 
for all $l \in \N_0$ it holds that 
\begin{equation}
\label{b45:eq041}
\begin{split}
 &\Big(
   [0,T]\times \R^d\times \Omega\ni (t,x,\omega)
   \mapsto 
          \left[
            \big(\funcF(U_{l,M}^{(0,l,i)})-\1_{\N}(l)\funcF( U_{l-1,M}^{(0,-l,i)})\big)
            \big(\uniform_t^{(0,l,i)},x+W_{\uniform_t^{(0,l,i)}}^{(0,l,i)}-W_t^{(0,l,i)}\big)
          \right] \! (\omega)
    \in \R
  \Big)_{i \in \Z}
\end{split}
\end{equation}
are continuous i.i.d.\ random fields.
Lemma~\ref{properties_approx}, the hypothesis that $(W^\theta)_{\theta \in \Theta}$ are i.i.d., \eqref{b45:eq04},
and Lemma~\ref{a08} therefore show that
for all $n \in \N$, $k \in \N_0$ it holds that
\begin{equation}
\label{b45:eq05}
\begin{split}
  &\left\|U_{n,M}^0-\E \!\left[U_{n,M}^0\right]\right\|_{k} \\
&\leq
   \left[  \textstyle \sum\limits_{i=1}^{M^{n}} \left| \tfrac{1}{M^n} \right|^2 \right]^{\nicefrac{1}{2}}
   \left\| [0,T]\times \R^d\times \Omega\ni (t,x,\omega)\mapsto \left[ g(x+W_T^{(0, 0, -1)}-W_t^{(0, 0, -1)}) \right] \! (\omega)\in\R\right\|_{k}
   \\
   &\quad+\sum_{l=0}^{n-1}
      \left[ \textstyle \sum_{i=1}^{M^{n-l}} \left| \tfrac{1}{M^{n-l}} \right|^2 \right]^{\nicefrac{1}{2}}
   \bigg\|[0,T]\times \R^d\times \Omega\ni (t,x,\omega) \\
&\qquad   \mapsto 
(T-t)    \left[ \big(\funcF(U_{l,M}^{(0,l,1)})-\1_{\N}(l)\funcF( U_{l-1,M}^{(0,-l,1)})\big) 
  (\uniform_t^{(0,l,1)} ,x+W_{\uniform_t^{(0,l,1)}  }^{(0,l,1)}-W_t^{(0,l,1)}) \right] \! (\omega)\in\R
\bigg\|_{k} \\
&=
   \tfrac{1}{\sqrt{M^{n}}}
   \left\| [0,T]\times \R^d\times \Omega\ni (t,x,\omega)\mapsto \left[ g(x+W_T^{(0, 0, -1)}-W_t^{(0, 0, -1)}) \right] \! (\omega)\in\R\right\|_{k} \\
&\quad
   +\sum\limits_{l=0}^{n-1} \tfrac{1}{\sqrt{M^{(n-l)}}}
   \bigg\|[0,T]\times \R^d\times \Omega\ni (t,x,\omega) \\
&\qquad 
   \mapsto 
(T-t)     \left[ \big(\funcF(U_{l,M}^{(0,l,1)})-\1_{\N}(l)\funcF( U_{l-1,M}^{(0,-l,1)})\big) 
  (\uniform_t^{(0,l,1)} ,x+W_{\uniform_t^{(0,l,1)}  }^{(0,l,1)}-W_t^{(0,l,1)}) \right] \! (\omega)\in\R
\bigg\|_{k}.
\end{split}
\end{equation}
Moreover, observe that Item~\eqref{a36:item1} in \cref{a36} and the hypothesis that $(W^\theta)_{\theta \in \Theta}$ and $\fwbm$ are independent assure that
for all $k \in \N_0$ it holds that
\begin{equation}
\begin{split}
\label{b45:eq06}
  &\left\| [0,T]\times \R^d\times \Omega\ni (t,x,\omega)\mapsto \left[g(x+W_T^{(0, 0, -1)}-W_t^{(0, 0, -1)}) \right] \! (\omega)\in\R\right\|_{k} 
=
  \tfrac{1}{ \sqrt{k!}} \! \left( \E\!\left[\left|g(\fwBM _T)\right|^2\right] \right)^{\! \nicefrac{1}{2}}.
\end{split}
\end{equation}
Furthermore, note that
the hypothesis that $(W^\theta)_{\theta \in \Theta}$ are i.i.d., 
the hypothesis that $(\uniform^\theta)_{\theta \in \Theta}$ are i.i.d., 
the hypothesis that $(W^\theta)_{\theta \in \Theta}$, $(\uniform^\theta)_{\theta \in \Theta}$, and $\fwbm$ are independent, 
Lemma~\ref{properties_approx},
and \cref{a26}  imply that
for all $n \in \N$, $k \in \N_0$ it holds that
\begin{equation}
\begin{split}
  &\sum_{l=0}^{n-1} 
    \tfrac{1}{\sqrt{M^{(n-l)}}}
    \bigg\|
      [0,T]\times \R^d\times \Omega\ni (t,x,\omega) \\
&\quad 
      \mapsto 
      (T-t) 
      \left[
        \big(\funcF(U_{l,M}^{(0,l,1)})-\1_{\N}(l)\funcF( U_{l-1,M}^{(0,-l,1)})\big) 
        (\uniform_t^{(0,l,1)} ,x+W_{\uniform_t^{(0,l,1)}  }^{(0,l,1)}-W_t^{(0,l,1)})
      \right] \! (\omega)\in\R
    \bigg\|_{k} \\
&\leq
  \sum_{l=0}^{n-1} 
    \tfrac{T}{\sqrt{M^{(n-l)}}} 
    \left\|
       \funcF(U_{l,M}^{(0,l,1)})-\1_{\N}(l)\funcF( U_{l-1,M}^{(0,-l,1)})  
    \right\|_{k+1}
\end{split}
\end{equation}
Item~\eqref{a07a:item1} in Lemma~\ref{a07a}, 
the hypothesis that $U^0_{0, M} = 0$,
and Lemma~\ref{a14} 
therefore demonstrate that
for all $n \in \N$, $k \in \N_0$ it holds that
\begin{equation}
\label{b45:eq07}
\begin{split}
  &\sum_{l=0}^{n-1} 
    \tfrac{1}{\sqrt{M^{(n-l)}}}
    \bigg\|
      [0,T]\times \R^d\times \Omega\ni (t,x,\omega) \\
&\quad 
      \mapsto 
      (T-t) 
      \left[
        \big(\funcF(U_{l,M}^{(0,l,1)})-\1_{\N}(l)\funcF( U_{l-1,M}^{(0,-l,1)})\big) 
        (\uniform_t^{(0,l,1)} ,x+W_{\uniform_t^{(0,l,1)}  }^{(0,l,1)}-W_t^{(0,l,1)})
      \right] \! (\omega)\in\R
    \bigg\|_{k} \\
&\leq
  \tfrac{T}{\sqrt{M^{n}}} 
  \left\|
       \funcF(U_{0,M}^{0})  
  \right\|_{k+1}
  +
  \sum_{l=1}^{n-1} 
    \tfrac{T}{\sqrt{M^{(n-l)}}} 
    \left(
      \left\|
         \funcF(U_{l,M}^{0})- F(u)
      \right\|_{k+1}
      +
      \left\|
        F(u)-\funcF( U_{l-1,M}^{0}) 
      \right\|_{k+1}
    \right) \\
&\leq
  \tfrac{T}{\sqrt{M^{n}}} 
  \left\|
       \funcF(0)  
  \right\|_{k+1}
  +
  \left[\sum_{l=1}^{n-1} 
    \tfrac{TL}{\sqrt{M^{(n-l )}}} 
      \left\|
         U_{l,M}^{0}- u
      \right\|_{k+1}
  \right]
  +
  \left[\sum_{l=1}^{n-1} 
      \tfrac{TL}{\sqrt{M^{(n-l)}}} 
      \left\|
        U_{l-1,M}^{0}  - u
      \right\|_{k+1} 
  \right]\\
&\leq
  \tfrac{T}{\sqrt{M^{n}}} 
  \left\|
       \funcF(0)  
  \right\|_{k+1}
  +
  \sum_{l=0}^{n-1} 
    \tfrac{(2-\1_{\{n-1\}}(l))LT}{\sqrt{M^{(n - l - 1)}}} 
      \left\|
         U_{l,M}^{0}- u
      \right\|_{k+1}.
\end{split}
\end{equation}
In addition, observe that \cref{a03} ensures that
for all $k\in \N_0$ it holds that
\begin{equation}
\|\funcF(0)\|_{k+1}^2
=\tfrac{1}{T^{k+1}}\int_0^T \tfrac{t^{k}}{k!}
\E\!\left[\left|\funcF(0)(t,\fwBM_t)\right|^2\right]\,dt
\leq \tfrac{T^k}{T^{k+1}k!}\int_0^T \E\!\left[\left|\funcF(0)(t,\fwBM_t)\right|^2\right]\,dt
=\tfrac{1}{k!}\|F(0)\|_{1}^2.
\end{equation}
Combining this \eqref{b45:eq05}, \eqref{b45:eq06},  and \eqref{b45:eq07} establishes that
for all $n \in \N$, $k\in\N_0$ it holds that
\begin{equation}  
\begin{split}
\label{b45:eq08}
   &\left\|U_{n,M}^0-\E[U_{n,M}^0]\right\|_{k} \\
&\leq 
  \tfrac{1}{\sqrt{k!M^{n}}}
  \left( \E\!\left[\left|g(\fwBM _T)\right|^2\right] \right)^{\! \nicefrac{1}{2}}
   +\tfrac{T}{\sqrt{k! M^{n}}}\|F(0)\|_{1}
   +\sum_{l=0}^{n-1}
       \tfrac{(2-\1_{\{n-1\}}(l))\LipConst T}{\sqrt{M^{(n-l-1)}}}
   \left\|
     U_{l,M}^{0}- u
   \right\|_{k+1} \\
&=
  \tfrac{1}{\sqrt{k!M^{n}}}
  \left[
    \left( \E\!\left[\left|g(\fwBM _T)\right|^2\right] \right)^{\! \nicefrac{1}{2}}
     + 
     T \|F(0)\|_{1}
  \right]
   +\sum_{l=0}^{n-1}
       \tfrac{(2-\1_{\{n-1\}}(l))\LipConst T}{\sqrt{M^{(n-l-1)}}}
   \left\|
     U_{l,M}^{0}- u
   \right\|_{k+1}   .
\end{split}     
\end{equation}
Next observe that, \eqref{b45:eq02}, \eqref{b45:eq03}, and \eqref{b45:eq041} demonstrate that
for all $n \in \N$, $t \in [0, T]$, $x \in \R^d$ it holds that
\begin{equation}
\label{b45:eq09}
\begin{split}
  &\E \!\left[U_{n,M}^0(t, x)\right] \\
&=
  \frac{1}{M^n}
  \left[ \sum_{i=1}^{M^n} 
    \Exp{ g\big(x+W^{0}_T-W^{0}_t\big)} 
  \right]
  +
  \sum_{l=0}^{n-1} \frac{(T-t)}{M^{n-l}}
    \left[\sum_{i=1}^{M^{n-l}}
      \E\!\left[ \left(\funcF( U_{l,M}^{0})-\1_{\N}(l)\funcF( U_{l-1,M}^{0})\right)\!(\uniform^0_t,x+W_{\uniform^0_t}^0-W_{t}^{0})\right]
    \right] \\
&=
  \E\!\left[g(x+W_{T}^{0}-W^{0}_{t})\right]+
  (T-t) 
  \Bigg(
    \sum_{l=0}^{n-1}  
    \E\!\left[ \funcF( U_{l,M}^{0})(\uniform^0_t,x+W_{\uniform^0_t}^0-W_{t}^{0})\right] 
    -
    \1_{\N}(l)\, \E\!\left[ \funcF( U_{l-1,M}^{0})(\uniform^0_t,x+W_{\uniform^0_t}^0-W_{t}^{0})\right]
  \Bigg)\\
&=
  \E\!\left[g(x+W^0_{T-t})\right]+
(T-t)   \E\!\left[ \left(\funcF( U_{n-1,M}^{0})\right)\!(\uniform^0_t,x+W^0_{\uniform^0_t}-W^0_{t})\right].
\end{split}
\end{equation}
In addition, note that \cref{b36b}, \cref{b36}, Fubini's theorem, and Lemma~\ref{contRF_eval_integrable} assure that
for all $t\in[0,T]$, $x\in \R^d$ it holds that
\begin{equation}
\begin{split}
  u(t,x)
&= 
  \Exp{
    g(x+\fwbm_{T-t}) 
  }
  +
  \int_t^{T}
    \Exp{
      (\funcF(u))(s,x+\fwbm_{s}-\fwbm_{t})
    }
  \,ds \\
&= 
  \Exp{
    g(x+\fwbm_{T-t}) 
  }
  +
  (T-t)
  \Exp{
      (\funcF(u))(\uniform^0_t,x+\fwbm_{\uniform^0_t}-\fwbm_{t})
  }\\
&= 
  \Exp{
    g(x+W^0_{T-t}) 
  }
  +
  (T-t)
  \Exp{
      (\funcF(u))(\uniform^0_t,x+W^0_{\uniform^0_t}-W^0_{t})
  }.
\end{split}
\end{equation}
Combining this with \cref{b45:eq09} yields that 
for all $n\in\N$, $t\in[0,T]$, $x\in \R^d$ it holds that
\begin{equation}
\begin{split}\label{c04b}
  &\E\!\left[U_{n,M}^{0}(t,x) \right] - u(t,x)  =
  (T-t)
  \left(
    \E\!\left[ \left(\funcF( U_{n-1,M}^{0})\right)\!(\uniform^0_t,x+W^0_{\uniform^0_t}-W^0_{t})\right]
    -
    \Exp{
      (\funcF(u))(\uniform^0_t,x+W^0_{\uniform^0_t}-W^0_{t})
  }
  \right) \\
&=  
  \E\!\left[(T-t) \left(\funcF( U_{n-1,M}^{0})-\funcF( u)\right)\!(\uniform^0_t,x+W^0_{\uniform^0_t}-W^0_{t}) \right].
\end{split}
\end{equation}
\cref{a07},
\cref{a26}, 
\cref{a14}, and
\cref{properties_approx} hence show that
for all $n\in \N$, $k \in \N_0$ it holds that
\begin{align}
\begin{aligned}
\left\|\E \!\left[U_{n,M}^0\right]-u\right\|_{k}
&=
  \Big\|
    [0,T]\times\R^d\times\Omega\ni (t,x,\omega)
    \mapsto \Exp{(T-t) \left(\funcF( U_{n-1,M}^{0})-\funcF( u)\right)\!(\uniform^0_t,x+W^0_{\uniform^0_t}-W^0_{t})} \in\R
  \Big\|_{k} \\
&\leq 
\Big\|
[0,T]\times\R^d\times\Omega\ni (t,x,\omega)
\mapsto (T-t) \left[\left(\funcF( U_{n-1,M}^{0})-\funcF( u)\right)\!(\uniform^0_t,x+W^0_{\uniform^0_t}-W^0_{t})\right](\omega) \in\R
\Big\|_{k}
\\&\leq T
\left\|\funcF( U_{n-1,M}^{0})-\funcF( u)\right\|_{k+1}
\leq \LipConst T
\left\|U_{n-1,M}^{0}- u\right\|_{k+1}.
\end{aligned}\label{f01}
\end{align}
This, \eqref{b45:eq01}, and \eqref{b45:eq08} demonstrate that 
for all $n \in \N$, $k\in \N_0$ it holds that
\begin{equation}
\label{b45:eq10}
\begin{split}
  \left\|U_{n,M}^0-u\right\|_{k}
&\leq 
  \tfrac{1}{\sqrt{k!M^{n}}}
  \left[
    \left( \E\!\left[\left|g(\fwBM _T)\right|^2\right] \right)^{\! \nicefrac{1}{2}}
     + 
     T \|F(0)\|_{1}
  \right] \\
&\qquad
   +
   \left[\sum_{l=0}^{n-1}
       \tfrac{(2-\1_{\{n-1\}}(l))\LipConst T}{\sqrt{M^{(n-l-1)}}}
       \left\|
         U_{l,M}^{0}- u
       \right\|_{k+1}
   \right]
   +
   \LipConst T
   \left\|U_{n-1,M}^{0}- u\right\|_{k+1}
\\
   &\leq \tfrac{1}{\sqrt{k!M^{n}}}\left[
\left(\E\!\left[\left|g(\start +\fwbm_T)\right|^2\right]\right)^{ \! \nicefrac{1}{2}}
   +T\left\|F(0)\right\|_1\right]
   +\sum_{l=0}^{n-1}
       \tfrac{2\LipConst T}{\sqrt{M^{(n-l-1)}}}
   \left\|U_{l,M}^{0}- u   \right\|_{k+1}.
\end{split}
\end{equation}
For the next step
let $\varepsilon_n\in [0,\infty]$, $n\in[0,N]\cap \N_0$, satisfy 
for all $n\in [0,N]\cap\N_0$ that
\begin{align}\label{f04}
\varepsilon_n =\sup\left\{ \tfrac{1}{\sqrt{M^j}} \left\| U_{n,M}^{0} - u  \right\|_{  k } \colon j,k\in\N_0,    j + n + k = N  \right\}
\end{align}
and
let $a_1,a_2 \in [0,\infty)$ be given by
\begin{align}\label{f03}
a_1=\sup_{k\in \{0,\ldots,N\}} \tfrac{1}{\sqrt{k!M^{N-k}}}\left[
\left(\E\!\left[\left|g(\start +\fwbm_T)\right|^2\right]\right)^{ \! \nicefrac{1}{2}}
   +T\left\|F(0)\right\|_1\right]
\qquad \text{and} \qquad 
a_2= 2LT.
\end{align} 
Observe that \eqref{b45:eq10} implies that
for all $n\in [1,N] \cap \N$, $j,k\in \N_0$ with  $j + n + k = N$ it holds that
\begin{equation}
\label{f02}
\begin{split}
&\tfrac{1}{\sqrt{M^j}}\left\|U_{n,M}^0-u\right\|_{k}  \leq \tfrac{1}{\sqrt{k!M^{n+j}}}\left[
\left(\E\!\left[\left|g(\start +\fwbm_T)\right|^2\right]\right)^{ \! \nicefrac{1}{2}}
   +T\left\|F(0)\right\|_1\right]
   +\sum_{l=0}^{n-1}
       \tfrac{2\LipConst T}{\sqrt{M^{(n+j-l-1)}}}
   \left\|U_{l,M}^{0}- u   \right\|_{k+1} \\
&= \tfrac{1}{\sqrt{k!M^{N-k}}}\left[
\left(\E\!\left[\left|g(\start +\fwbm_T)\right|^2\right]\right)^{ \! \nicefrac{1}{2}}
   +T\left\|F(0)\right\|_1\right]
   +\sum_{l=0}^{n-1}
       \tfrac{2\LipConst T}{\sqrt{M^{(N-k-l-1)}}}
   \left\|U_{l,M}^{0}- u   \right\|_{k+1} \leq
 a_1+a_2\sum_{l=0}^{n-1} \varepsilon_l.
\end{split}
\end{equation}
Hence, we obtain 
for all $n\in [1,N] \cap \N$ that
$
\varepsilon_n\leq a_1+a_2\sum_{l=0}^{n-1} \varepsilon_l=(a_1+a_2 \varepsilon_0)+a_2\sum_{l=1}^{n-1}\varepsilon_l
$.
The discrete Gronwall-type inequality in \cite[Corollary 4.1.2]{agarwal2000difference} hence proves that 
for all $n\in [1,N] \cap \N$ it holds that
$
\varepsilon_n \leq (a_1+a_2 \varepsilon_0)(1+a_2)^{n-1}
$.
This, \eqref{a03}, and \eqref{f04} imply that
\begin{equation}\label{f05c}
\begin{split}
\left(\E\!\left[\left|U^0_{N,M}(0,\start)-u(0,\start)\right|^2\right]\right)^{\!\!\nicefrac{1}{2}}
&= \left\|U^0_{N,M}-u\right\|_{0} 
=\varepsilon_N \leq (a_1+a_2 \varepsilon_0)(1+a_2)^{N-1}
\leq \max\{a_1,\varepsilon_0\}(1+a_2)^{N}.
\end{split}
\end{equation}
Moreover, observe that
\begin{equation}\label{b48}
    \sup_{k\in \{0, \ldots, N \} }\tfrac{1}{M^{(N-k)}k!}=\frac{1}{M^{N}}
     \sup_{k\in\{0, \ldots, N \}}\tfrac{M^{k}}{k!}
    \leq \frac{1}{M^{N}}\sum_{k=0}^\infty\tfrac{M^{k}}{k!}=\frac{e^{M}}{M^{N}}.
    \end{equation}
Therefore, we obtain that
\begin{align}\label{f06}
\begin{aligned}
a_1
&\leq  
\left[
\left(\E\!\left[\left|g(\start +\fwbm_T)\right|^2\right]\right)^{\!\!\nicefrac{1}{2}}+T\left\|F(0)\right\|_1  \right]
\frac{e^{\nicefrac{M}{2}}}{M^{\nicefrac{N}{2}}}.
\end{aligned}
\end{align}
In addition, note that the hypothesis that $U^0_{0,M} = 0$, Item~\eqref{a36:item2} in Lemma~\ref{a36}, \eqref{b48}, and Lemma~\ref{upper_exact}
ensure that
\begin{equation}  \begin{split}\label{c02}
  \varepsilon_0=\sup_{k\in \{0,\ldots,N\}}\frac{\|u\|_{k}}{\sqrt{M^{(N-k)}}}
&\leq 
\left[\sup_{t\in [0,T]}\left(\E\!\left[\left|u(t,\start + \fwbm_t)\right|^2\right]\right)^{\!\nicefrac{1}{2}}\right]
\left[\sup_{k\in \{0,\ldots,N\}}\tfrac{1}{\sqrt{M^{(N-k)}k!}}\right]\\
&\leq
e^{LT}\left[
    \left(\E\!\left[\left|g(\xi+\fwbm_T)\right|^2\right]\right)^{\!\nicefrac{1}{2}}
    +T\|F(0)\|_1\right]
\frac{e^{\nicefrac{M}{2}}}{M^{\nicefrac{N}{2}}}  .
\end{split}     \end{equation}
This and \cref{f06} assure that 
\begin{align}
\max\{a_1,\varepsilon_0\}
\leq
e^{LT}\left[
    \left(\E\!\left[\left|g(\xi+\fwbm_T)\right|^2\right]\right)^{\!\nicefrac{1}{2}}
    +T\|F(0)\|_1\right]
\frac{e^{\nicefrac{M}{2}}}{M^{\nicefrac{N}{2}}}  .
\end{align}
Combining this with
\cref{f03,f05c} establishes that
\begin{equation} 
\begin{split}
&\left(\E\!\left[\left|U^0_{N,M}(0,\start)-u(0,\start)\right|^2\right]\right)^{\!\nicefrac{1}{2}}
\leq
e^{LT}
\left[\left(\E\!\left[\left|g(\start +\fwbm_T)\right|^2\right]\right)^{ \! \nicefrac{1}{2}}+T\left\|\funcF(0)\right\|_1  \right]\frac{e^{\nicefrac{M}{2}}(1+2LT)^{N}}{M^{\nicefrac{N}{2}}} 
.
\end{split}
\end{equation}
The proof of \cref{b45} is thus completed.
\end{proof}

\subsection{Analysis of the computational effort}\label{t5}
In this subsection we combine the error analysis provided in Subsection~\ref{t4} for the in Subsection~\ref{t1} introduced MLP approximation algorithms with a computational cost analysis to establish in Theorem~\ref{thm1} the overall complexity analysis for the proposed MLP approximation algorithms. 

In \cref{c15} below, for every $n,M\in \N$ we think of $\RN_{n,M}$ as an upper bound for the sum of the number of realizations of scalar standard normal random variables which are required to compute one realization of $U^0_{n,M}(0,0)$ in \eqref{thm1:ass2} below
and the number of realizations of on $[0,1]$ uniformly distributed random variables which are required to compute one realization of $U^0_{n,M}(0,0)$ in \eqref{thm1:ass2} below.
Roughly speaking, for every $n,M\in \N$ one realization of $U^0_{n,M}(0,0)$ employs $dM^n$ realizations of scalar standard normal random variables to calculate the second of the two summands on the right-hand side of \eqref{thm1:ass2} (the Monte Carlo sum involving the terminal condition $g\colon \R^d \to \R$). 
 Additionally, roughly speaking, for every $n,M\in \N$, $l\in \{0,1,\ldots, n-1\}$ one realization of $U^0_{n,M}(0,0)$ employs $dM^{n-l}$ scalar standard normal random variables and $M^{n-l}$ on $[0,1]$ uniformly distributed random variables to evaluate the $(l+1)$-th summand within the first of the two summands on the right-hand side of \eqref{thm1:ass2} (the Monte Carlo sum involving the difference of the nonlinearity $f\colon [0,T]\times \R^d\times \R \to \R$).
Moreover, roughly speaking, for every $n,M\in \N$, $l\in \{0,1,\ldots, n-1\}$ one realization of $U^0_{n,M}(0,0)$ employs $M^{n-l}$ realizations of $U^\theta_{l,M}(t,x)$ and $\1_{ \N }( l )M^{n-l}$ realizations of $U^\theta_{l-1,M}(t,x)$ for some suitable $\theta\in \Theta$, $t\in [0,T]$, and $x\in \R^d$. Note that for every $n,M\in \N$, $\theta\in \Theta$, $t\in [0,T)$, $x\in \R^d$ it holds that the number of realizations of scalar random variables required to compute one realization of $U^\theta_{n,M}(t,x)$ is equal to the number of realizations of scalar random variables required to compute one realization of $U^0_{n,M}(0,0)$.

\begin{lemma}[Computational effort]\label{c15}
Let $d \in \N$ and
$(\RN_{n,M})_{n,M\in \Z}\subseteq\N_0$ satisfy
for all $n,M \in \N$ that 
$\RN_{0,M}=0$
and 
\begin{align}
\label{c16}
  \RN_{ n,M}
  &\leq d M^n+\sum_{l=0}^{n-1}\left[M^{(n-l)}( d+1 + \RN_{ l, M}+ \1_{ \N }( l )  \RN_{ l-1, M })\right].
\end{align}
Then 
it holds 
for all $n, M\in\N$ that
$
\RN_{ n, M }
\leq d \,(5M)^n
$.
\end{lemma}

\begin{proof}[Proof of \cref{c15}]
First, observe that~\eqref{c16} and the fact that 
for all $ M \in \N$ it holds that 
$\RN_{0,M}=0$ imply that
for all $ n \in \N $, $M \in \N\cap [2,\infty)$ it holds 
that
\begin{equation}  
\begin{split}\label{c17}
  (M^{-n} \RN_{ n,M})
&\leq 
  d+\sum_{l=0}^{n-1}\left[ M^{-l}( d +1 + \RN_{ l, M}+ \1_{ \N }( l )  \RN_{ l-1, M })\right]\\
&\leq 
  d+ (d+1) \left[ \sum_{l=0}^{n-1} M^{-l} \right]  + \left[ \sum_{l=0}^{n-1} M^{-l} \RN_{ l, M} \right] + \left[ \sum_{l=0}^{n-2} M^{-(l+1)} \RN_{ l, M } \right] \\
&=
  d+ (d+1) \tfrac{(1 - M^{-n})}{(1- M^{-1})} + \left[ \sum_{l=0}^{n-1} M^{-l} \RN_{ l, M} \right] +  \frac{1}{M} \left[ \sum_{l=0}^{n-2} M^{-l} \RN_{ l, M } \right] \\
&\leq
  d+ (d+1) \tfrac{1}{(1- \frac{1}{2})} + \left( 1 + \tfrac{1}{M}\right) \left[ \sum_{l=0}^{n-1} M^{-l} \RN_{ l, M} \right] \\
&= 
  3d + 2 + \left( 1 + \tfrac{1}{M}\right) \left[ \sum_{l=1}^{n-1} M^{-l} \RN_{ l, M} \right].
\end{split}     
\end{equation}
The discrete Gronwall-type inequality in \cite[Corollary 4.1.2]{agarwal2000difference} hence ensures that 
for all $ n \in \N $, $M \in \N\cap [2,\infty)$ it holds that
\begin{equation}
 (M^{-n} \RN_{ n,M })
\leq
  (3d+2)\left(2+\tfrac1M\right)^{n-1}.
\end{equation}
This establishes that 
for all $ n \in \N $, $M \in \N\cap [2,\infty)$ it holds that
\begin{equation}
\label{c18}
 \RN_{ n,M }
\leq
  (3d+2)\left(2+\tfrac1M\right)^{n-1}M^n
\leq (5d) 3^{n-1}M^n
\leq d (5M)^n.
\end{equation}
Moreover, observe that the fact that $\RN_{0, 1} = 0$ and \eqref{c16} demonstrate that 
for all $n \in \N$ it holds that
\begin{equation}
  \RN_{n, 1}
\leq
  d
  +
  \sum_{l = 0}^{n-1} (d+1 + \RN_{l, 1} + \mathbbm{1}_{\N}(l)\RN_{l-1, 1})
\leq
  d + n(d+1) + 2 \sum_{l = 1}^{n-1}\RN_{l, 1}.
\end{equation}
Hence, we obtain 
for all $n \in \N$, $k \in \N \cap (0,n]$ that
$
  \RN_{k, 1}
\leq
  d + n(d+1) + 2 \sum_{l = 1}^{k-1}\RN_{l, 1}
$.
Combining this with the discrete Gronwall-type inequality in \cite[Corollary 4.1.2]{agarwal2000difference} proves that 
for all $n \in \N$, $k \in \N \cap (0,n]$ it holds that
$
  \RN_{k, 1}
\leq
  (d + n(d+1)) 3^{k-1}
$.
The fact that 
for all $n \in \N$ it holds that
$(1 + 2n) 3^{n-1} \leq 5^n$
hence
shows that 
for all $n \in \N$ it holds that
\begin{equation}
  \RN_{n, 1}
\leq
  (d + n(d+1)) 3^{n-1}
=
  d \, \left(1 + n\left(1+\tfrac{1}{d}\right)\right) 3^{n-1}
\leq
  d \, (1 + 2n) 3^{n-1}
\leq
  d \, 5^n.
\end{equation}
Combining this with \eqref{c18} completes the proof of \cref{c15}.
\end{proof}

\begin{corollary}
\label{comp_and_error}
Assume \cref{b35}, assume that $f$ and $g$
are at most polynomially growing, and let $\delta \in (0,\infty)$, $C \in (0,\infty]$, $\epsilon \in (0,1]$,
$(\RN_{n,M})_{n,M \in \Z}\subseteq\N_0$ satisfy
for all $n,M \in \N$ that 
\begin{align}
\label{comp_and_error:ass1}
\RN_{0,M}=0,
\qquad
  \RN_{ n,M}
  &\leq d M^n+\sum_{l=0}^{n-1}\left[M^{(n-l)}( d+1 + \RN_{ l, M}+ \1_{ \N }( l )  \RN_{ l-1, M})\right], \qquad \text{and}
\end{align}
\begin{equation}
  C 
=
\max\left\{100, 5 e \left[e^{LT}
  \left[\left(\E\!\left[\left|g(\start + \fwbm_T)\right|^2\right]\right)^{\!\!\nicefrac{1}{2}}+T\left\|\funcF(0)\right\|_1  \right] \right]^{2+\delta} \sup_{n\in \N}\left[
\frac{(n+1)(4+8LT)^{n(2+\delta)}}{n^{((n\delta)/2)}}
 \right]\right\}.
\end{equation}
Then 
\begin{enumerate}[(i)]
\item \label{item:C_finite} it holds that $C<\infty$ and 
\item \label{item:complexity} there exists $N\in \N$ such that $\RN_{N,N}
\leq C
  d 
    \epsilon^{-(2 + \delta)}
  $
  and
$
\sup_{n\in \N\cap [N,\infty)}\left(\E\!\left[|U^0_{n,n}(0,\start)-u(0,\start)|^2\right]\right)^{\!\nicefrac{1}{2}}\le \epsilon
$.
\end{enumerate}
\end{corollary}

\begin{proof}[Proof of Corollary~\ref{comp_and_error}]
Throughout this proof let $c \in (0,\infty]$, $\kappa \in (0,\infty)$ satisfy
\begin{equation}
  c
=
   e^{LT}
  \left[\left(\E\!\left[\left|g(\start + \fwbm_T)\right|^2\right]\right)^{\!\!\nicefrac{1}{2}}+T\left\|\funcF(0)\right\|_1  \right]
\qquad \text{and} \qquad 
  \kappa 
=
  \sqrt{e}(1+2LT).
\end{equation}
Note that the fact that for all 
$ p \in (0,\infty) $ it holds that
$
  \E\big[
    \sup\nolimits_{ t \in [0,T] }
    \| W^0_t \|^p_{ \R^d }
  \big]
  < \infty
$
and the hypothesis that 
$ f \colon [0,T] \times \R^d \times \R \to \R $
and 
$ g \colon \R^d \to \R $
are at most polynomially growing functions 
ensure that $c<\infty$. This establishes Item~\eqref{item:C_finite}. 
Next observe that Theorem~\ref{b45} (with $N \dashleftarrow n$, $M \dashleftarrow n$ in the notation of Theorem~\ref{b45}) ensures that
for all $n\in \N$ it holds that
\begin{equation}
   \left(\E\!\left[|U^0_{n,n}(0,\start)-u(0,\start)|^2\right]\right)^{\!\!\nicefrac{1}{2}}
\leq
  \frac{c \, e^{n/2}(1+2LT)^{n}}{n^{n/2}}
=
  \frac{c \kappa^n}{n^{\nicefrac{n}{2}}}.
\end{equation}
This proves that $\limsup_{n\to \infty}\left(\E\!\left[|U^0_{n,n}(0,\start)-u(0,\start)|^2\right]\right)^{\!\nicefrac{1}{2}}=0$. Next let $N\in \N\cap [2,\infty)$ satisfy 
\begin{equation}\label{eq:def_N}
N=\min\left\{n\in \N\cap [2,\infty): \sup_{m\in \N\cap [n,\infty)}\left(\E\!\left[|U^0_{n,n}(0,\start)-u(0,\start)|^2\right]\right)^{\!\nicefrac{1}{2}}\le \varepsilon\right\}.
\end{equation}
Note that \eqref{eq:def_N} implies that
\begin{equation}\label{eq:ub_eps}
\varepsilon \le \1_{\{2\}}(N)+\left(\E\!\left[|U^0_{n,n}(0,\start)-u(0,\start)|^2\right]\right)^{\!\nicefrac{1}{2}}\1_{[3,\infty)}(N)
\le \1_{\{2\}}(N)+ \frac{c \kappa^{N-1}\1_{[3,\infty)}(N)}{(N-1)^{\nicefrac{(N-1)}{2}}}.
\end{equation}
Next note that the fact that $\sup_{n\in \N}[(1+\frac{1}{n})^n]=e$ ensures that
\begin{equation}
\frac{N^N}{(N-1)^{N-1}}=\frac{N^{N-1}N}{(N-1)^{N-1}}=N\left(1+\frac{1}{N-1}\right)^{N-1}\le Ne.
\end{equation}
Lemma~\ref{c15} and \cref{eq:ub_eps} hence demonstrate that
\begin{equation}
\label{c19}
\begin{split}
  \RN_{N,N}
&\leq
  d (5N)^N
=
  d (5N)^N \varepsilon^{2+\delta} \varepsilon^{-(2+\delta)}
\leq
d \varepsilon^{-(2+\delta)} 
\left(100\, \1_{\{2\}}(N)+ (5N)^N\left(\frac{c \kappa^{N-1}\1_{[3,\infty)}(N)}{(N-1)^{\nicefrac{(N-1)}{2}}}\right)^{2+\delta}\right)\\
&
\leq
d \varepsilon^{-(2+\delta)} 
\max\left\{100,(5N)^N\left(\frac{c \kappa^{N-1}}{(N-1)^{\nicefrac{(N-1)}{2}}}\right)^{2+\delta}\right\}\\
&\leq d \varepsilon^{-(2+\delta)} 
\max\left\{100, \frac{c^{2+\delta}5^N\kappa^{(N-1)(2+\delta)}N^N}{(N-1)^{N-1}(N-1)^{\frac{\delta(N-1)}{2}}}
\right\}
\leq d \varepsilon^{-(2+\delta)} 
\max\left\{100, \frac{c^{2+\delta}eN5^N\kappa^{(N-1)(2+\delta)}}{(N-1)^{\frac{\delta(N-1)}{2}}}
\right\}
\\
&\leq d \varepsilon^{-(2+\delta)} 
\max\left\{100, 5 e c^{2+\delta} \sup_{n\in \N}\left[
\frac{(n+1)5^n\kappa^{n(2+\delta)}}{n^{((n\delta)/2)}}
 \right]\right\}.
\end{split}
\end{equation}
In addition, observe that the fact that $\sqrt{5e} \leq 4$ assures that
\begin{equation}
  5 \kappa^{(2+\delta)}
\leq
  ( \sqrt{5 e} (1+2LT) )^{(2+\delta)}
\leq
  ( 4 (1+2LT) )^{(2+\delta)}
=
  ( 4 +8LT )^{(2+\delta)}.
\end{equation}
This and \eqref{c19} prove that
\begin{equation}
  \RN_{N,N}
\leq
  d \varepsilon^{-(2+\delta)} 
\max\left\{100, 5 e c^{2+\delta} \sup_{n\in \N}\left[
\frac{(n+1)(4+8LT)^{n(2+\delta)}}{n^{((n\delta)/2)}}
 \right]\right\}
=
  Cd \varepsilon^{-(2+\delta)}.
\end{equation}
This establishes Item~\eqref{item:complexity}.
The proof of Corollary~\ref{comp_and_error} is thus completed.
\end{proof}

\begin{theorem}
\label{thm1}
Let $d \in \N$, $L, T, \delta \in (0,\infty)$, $\varepsilon \in (0,1]$, $C \in (0, \infty]$, $\xi \in \R^d$, $\Theta = \cup_{n = 1}^\infty \Z^n$, 
let 
$f \in C( [0,T] \times \R^d \times \R , \R)$,
$g  \in C (\R^d , \R)$,
$u \in C( [0,T] \times \R^d , \R)$
be at most polynomially growing functions,
let $ ( \Omega, \mathcal{F}, \P ) $ be a probability space, 
let $W^{\theta} \colon [0,T] \times \Omega \to \R^d$, $\theta \in \Theta$, be independent standard Brownian motions with continuous sample paths,
let $\unif^\theta\colon \Omega\to[0,1]$, $\theta\in \Theta$, be independent $\mathcal{U}_{[0,1]}$-distributed random variables, 
assume that $(W^\theta)_{\theta \in \Theta}$ and $(\unif^\theta)_{\theta \in \Theta}$ are independent,
assume 
for all $t \in [0, T]$, $x \in \R^d$, $v, w \in \R$ that
\begin{equation}
\label{eq:fixed_point_cost_thm}
  u(t,x)
= 
  \Exp{
    g(x+W^0_{T-t}) 
    +
    \textstyle \int_t^{T} \!
      f \big(s,x+W^0_{s-t}, u(s,x+W^0_{s-t}) \big)
    \,ds
  },
\end{equation}
\begin{multline}
  C  
= 
\max\Bigg\{100, 5 e 
 \left[
    e^{LT} \!
    \left(\left(\E \big[|g(\start + W^0_T)|^2 \big]\right)^{\!\!\nicefrac{1}{2}}
    +
    \sqrt{T} \,
    \big| \! \textstyle\int_{0}^T \displaystyle \Exp{ | f( s , \xi + W^0_s, 0 ) |^2 }  ds  \big|^{ \nicefrac{1}{2}}
    \right)
  \right]^{2 + \delta}
  \\
  \sup_{n\in \N}\left[
\frac{(n+1)(4+8LT)^{n(2+\delta)}}{n^{((n\delta)/2)}}
 \right]\Bigg\},
\end{multline}
and
$
  |f(t, x, v) - f(t, x, w)|
\leq
  L | v - w |
$,
let $\uniform^\theta\colon [0,T]\times \Omega\to [0, T]$, $\theta\in\Theta$, 
and
$ 
  {U}_{ n, M}^{\theta} \colon [0,T] \times \R^d \times \Omega \to \R
$, $n, M \in\Z$, $\theta\in\Theta$,
satisfy 
for all $ n, M \in \N$, $\theta\in\Theta $, 
$ t \in [0,T]$, $x\in\R^d $
that  
$\uniform^\theta _t = t+ (T-t)\unif^\theta$,
$U_{-1, M}^{\theta}(t,x)=
U_{0, M}^{\theta}(t,x)=0$, and 
\begin{equation}
\begin{split}
\label{thm1:ass2}
    &U_{n,M}^{\theta}(t,x)
  =
  \Bigg[
  \sum_{l=0}^{n-1}\tfrac{(T-t)}{M^{n-l}}\sum_{i=1}^{M^{n-l}}
  f \Big(
    \uniform_t^{(\theta,l,i)},
    x+W_{\uniform_t^{(\theta,l,i)} - t}^{(\theta,l,i)}, 
    U_{l,M}^{(\theta,l,i)}(\uniform_t^{(\theta,l,i)},x+W_{\uniform_t^{(\theta,l,i)} - t}^{(\theta,l,i)} )
  \Big) \\
  &-
  \1_{\N}(l)
  f \Big(
    \uniform_t^{(\theta,l,i)},
    x+W_{\uniform_t^{(\theta,l,i)} - t}^{(\theta,l,i)}, 
    U_{l-1,M}^{(\theta,-l,i)}(\uniform_t^{(\theta,l,i)},x+W_{\uniform_t^{(\theta,l,i)} - t}^{(\theta,l,i)})
  \Big)
  \Bigg]
  +
  \sum_{i=1}^{M^n} \frac{g ( x+W^{(\theta,0,-i)}_{T-t} )}{M^n},
\end{split}
\end{equation}
and let
$(\RN_{n, M})_{n, M \in\N_0 }\subseteq\N_0$ satisfy
for all $n, M\in \N$ that 
$\RN_{0, M}=0$
and 
\begin{align}
\label{thm1:ass1}
  \RN_{ n, M}
  &\leq d M^n+\sum_{l=0}^{n-1}\left[M^{(n-l)}( d+1 + \RN_{ l, M}+ \1_{ \N }( l )  \RN_{ l-1, M})\right].
\end{align} 
Then
\begin{enumerate}[(i)]
\item \label{thm1:item2}
it holds that $C < \infty$
and

\item \label{thm1:item3}
there exists $N\in \N$ such that $\RN_{N,N}
\leq
  C d
    \epsilon^{-(2 + \delta)}
  $
  and
$
\sup_{n\in \N\cap [N,\infty)}\left(\E\!\left[|U^0_{n,n}(0,\start)-u(0,\start)|^2\right]\right)^{\!\nicefrac{1}{2}}\le \epsilon
$.

\end{enumerate}
\end{theorem}

\begin{proof}[Proof of Theorem~\ref{thm1}]
Throughout this proof let $ F \colon C( [0,T] \times \R^d, \R ) \to C( [0,T] \times \R^d , \R ) $ 
satisfy
for all 
$ v \in C( [0,T] \times \R^d, \R ) $,
$ t \in [0,T] $, $ x \in \R^d $ that
\begin{equation}
  ( F( v ) )( t, x )
  =
  f( t, x, v(t,x) )
  .
\end{equation}
Observe that the hypothesis that for all 
$ t \in [0,T] $, $ x \in \R^d $, $ v, w \in \R $
it holds that
$
  | f( t, x, v ) - f( t, x, w ) |
  \leq 
  L | v - w |
$
ensures that
for all $ v, w \in C([0, T] \times \R^d, \R)$, $t\in [0,T]$, 
$x\in  \R^d$ 
it holds that 
\begin{equation}
\label{eq:verify_Lipschitz}
\begin{split}
    |
      (F(v))(t,x) -(F(w))(t,x)
    |
  = 
    | 
      f( t, x, v(t,x) ) - f( t, x, w(t,x) )
    |
  \leq
    \LipConst
    \left| v(t,x) - w(t,x) \right|
  .
\end{split}     
\end{equation}
Moreover, note that 
the hypothesis that 
$ f \colon [0,T] \times \R^d \times \R \to \R $, 
$ g \colon \R^d \to \R $,
and 
$
  u \colon [0,T] \times \R^d \to \R 
$
are at most polynomially growing functions 
and the fact that for all 
$ p \in (0,\infty) $ it holds that
$
  \E\big[
    \sup\nolimits_{ t \in [0,T] }
    \| W^0_t \|^p_{ \R^d }
  \big]
  < \infty
$
demonstrate that
\begin{equation}
\label{eq:verify_polynomial_growth}
\begin{split}
  &\E\big[
    | g( x + W^0_t ) |
  \big]
+
  \int_0^T 
    \big(
      \E\big[
        | u(s,\xi+ W^0_s ) |^2
      \big]
    \big)^{\!\nicefrac{1}{2}}
  \,ds \\
&
  +
  \int_t^T 
  \E\big[
    |
      (\funcF(u))(s,x+W^0_{ s - t } )
    |
    +
    |
      (\funcF(0))(s, x + W^0_{ s - t  } )
    |
  \big] \, ds
  < \infty
  .
\end{split}
\end{equation}
Combining this and \cref{eq:verify_Lipschitz} 
with \cref{comp_and_error}
establishes Item~\eqref{thm1:item2}. 
In addition, observe that
\cref{eq:verify_Lipschitz},
\cref{eq:verify_polynomial_growth}, 
and 
\cref{comp_and_error}
establish Item~\eqref{thm1:item3}. 
The proof of Theorem~\ref{thm1} is thus completed.
\end{proof}

\subsubsection*{Authors' contributions}
All authors made substantial contributions to the conception and the design of this work and all authors also substantially contributed to the drafting and the revisions of this work. Moreover, each of the authors gave his final approval for the publication of the final version of this article. In addition, each of the authors agrees to be accountable for all aspects of this work in ensuring that questions related to the accuracy or the integrity of any part of this work are appropriately investigated and resolved. Furthermore, all authors confirm that there is no other person who substantially contributed to this work.

\subsubsection*{Acknowledgements}
This work has been funded by the Deutsche Forschungsgemeinschaft (DFG, German Research Foundation) under Germany’s Excellence Strategy EXC 2044-390685587, Mathematics M\"unster:  Dynamics-Geometry-Structure and through the research grant HU1889/6-1. Christian Beck, Ramon Braunwarth, and Emilia Magnani are gratefully acknowledged for bringing a few typos into our notice.

{\small
\bibliographystyle{acm}
\bibliography{bibfile}
}

\end{document}